\documentclass [11pt,a4paper]{article}
\usepackage[ansinew]{inputenc}
\usepackage{amssymb}
\usepackage{amsmath}
\usepackage{graphicx}
\usepackage{amsthm}
\usepackage{cite}

\newtheorem{theorem}{\textbf{Theorem}}[section]
\newtheorem{lemma}{\textbf{Lemma}}[section]
\newtheorem{proposition}{\textbf{Proposition}}[section]
\newtheorem{corollary}{\textbf{Corollary}}[section]
\newtheorem{remark}{\textbf{Remark}}[section]
\newtheorem{definition}{\textbf{Definition}}[section]

\allowdisplaybreaks[4]

\def\be{\begin{equation}}
\def\ee{\end{equation}}
\def\bea{\begin{eqnarray}}
\def\eea{\end{eqnarray}}
\def\bt{\begin{theorem}}
\def\et{\end{theorem}}
\def\bl{\begin{lemma}}
\def\el{\end{lemma}}
\def\br{\begin{remark}}
\def\er{\end{remark}}
\def\bp{\begin{proposition}}
\def\ep{\end{proposition}}
\def\bc{\begin{corollary}}
\def\ec{\end{corollary}}
\def\bd{\begin{definition}}
\def\ed{\end{definition}}
\def\la{\lambda}
\def\non{\nonumber }

\begin{document}

\title{On the General Ericksen--Leslie System: Parodi's Relation, Well-posedness and Stability}
\author{
{\sc Hao Wu} \footnote{School of Mathematical Sciences and Shanghai
Key Laboratory for Contemporary Applied Mathematics, Fudan
University, 200433 Shanghai, China, Email:
\textit{haowufd@yahoo.com}.}, {\sc Xiang Xu} \footnote{Department of
Mathematical Sciences, Carnegie Mellon University, Pittsburgh, PA
15213, Email: \textit{xuxiang@math.cmu.edu}.} \ and {\sc Chun Liu}
\footnote{Department of Mathematics, Penn State University, State
College, PA 16802, Email: \textit{liu@math.psu.edu}.} }

\date{\today}

\maketitle

\begin{abstract}
In this paper we investigate the role of Parodi's relation
in the well-posedness and stability of the general Ericksen--Leslie
system modeling nematic liquid crystal flows. First, we give a formal physical derivation of the
Ericksen--Leslie system through an appropriate energy variational approach under Parodi's relation, in
which we can distinguish the conservative/dissipative parts of the
induced elastic stress. Next, we prove global well-posedness and long-time
behavior of the Ericksen--Leslie system under the assumption that the viscosity $\mu_4$ is sufficiently large.
Finally, under Parodi's
relation, we show the global well-posedness and Lyapunov stability for the Ericksen--Leslie system near local energy minimizers.
The connection between Parodi's relation and linear stability of the Ericksen--Leslie system is also discussed.\smallskip

\noindent \textbf{Keywords}: Liquid crystal flows, Ericksen--Leslie System, Parodi's
relation, uniqueness of asymptotic limit, stability. \\
\textbf{AMS Subject Classification}: 35B40, 35B41, 35Q35, 76D05.
\end{abstract}

\section{Introduction}

Liquid crystal is often viewed as the fourth state of the matter
besides the gas, liquid and solid, or as an intermediate state
between liquid and solid. It possesses none or partial positional
order but displays an orientational order at the same time. The
nematic phase is the simplest among all liquid crystal phases and is
close to the liquid phase. The molecules float around as in a liquid
phase, but have the tendency of aligning along a preferred direction
due to their orientation. The hydrodynamic theory of liquid crystals
due to Ericken and Leslie was developed around 1960's
\cite{E61,E62,Le66,Le68}. Earlier attempts on rigorous mathematical
analysis of the Ericksen--Leslie system were made recently
\cite{LL01} (see \cite{Lin89,LL95,LL96} for a simplified system
which carried important mathematical difficulties of the original
Ericksen--Leslie system, except the kinematic transport of the
director field).

The full Ericksen--Leslie system consists of the
following equations (cf. \cite{E87, Le68, Le79,LL01}):
 \begin{eqnarray}
&& \rho_t+v\cdot\nabla\rho = 0, \label{the most primitive equ 1}\\
&& \rho \dot{v}=\rho F+\nabla \cdot \hat{\sigma}, \label{the most primitive equ 2} \\
&& \rho_1\dot{\omega} = \rho_1 G+\hat{g}+\nabla \cdot \pi. \label{the most primitive equ 3}
 \end{eqnarray}
 Equations \eqref{the most
primitive equ 1}--\eqref{the most primitive equ 3} represent the
conservation of mass, linear momentum, and angular momentum,
respectively, with the anisotropic feature of liquid crystal
materials exhibited in \eqref{the most primitive equ 3} and its
nonlinear coupling in \eqref{the most primitive equ 2} (cf.
\cite{Le68,LL01}). In this paper, we consider the flow of an
incompressible material, namely, $\nabla \cdot v=0$. Here, $\rho$ is
the fluid density, $\rho_1$ is a (positive) inertial constant,
 $v=(v_1, v_2, v_3)^{T}$ is the flow velocity, $d=(d_1, d_2, d_3)^{T}$ is the orientational
 order parameter representing the macroscopic average of the
molecular directors, $\hat{g}$ is the intrinsic force associated
with $d$, $\pi$ is the director stress, $F$ and $G$ are external
body force and external director body force, respectively. The
superposed dot denotes the material derivative $\partial_t + v\cdot
\nabla$. The notations
 \bea
 && A=\frac12(\nabla v+\nabla^{T}v),\
\quad\ \  \quad \quad \Omega=\frac12(\nabla v-\nabla^{T}v), \non\\
&& \omega=\dot{
d}=d_t+( v\cdot\nabla)d, \ \,\; \quad\quad  N=\omega-\Omega\,d,
\non
 \eea
 represent the
rate of strain tensor, skew-symmetric part of the strain rate, the
material derivative of $d$ (transport of center of mass) and rigid
rotation part of director changing rate by fluid vorticity,
respectively.

We have the following constitutive relations in the system \eqref{the most
primitive equ 1}--\eqref{the most primitive equ 3} for
$\hat{\sigma}$, $\pi$ and $\hat{g}$:
\begin{eqnarray}
\hat{\sigma}_{ij}&=&-P\delta_{ij}-\rho\frac{\partial W}{\partial
d_{k,i}}d_{k,j}+\sigma_{ij}, \label{v1}\\
\pi_{ij}&=&\beta_id_j+\rho\frac{\partial W}{\partial d_{j,
i}}, \\
\hat{g}_{i}&=&\gamma d_i-\beta_jd_{i,j}-\rho\frac{\partial
W}{\partial d_{i}}+g_{i}.
\end{eqnarray}
Here $P$ is a scalar function representing the pressure. The vector
$\beta=(\beta_1,\beta_2, \beta_3)^T$  and the scalar function
$\gamma$ (sometimes called director tension) are Lagrangian
multipliers for the constraint on the length of director such that
$|d|=1$, with the Oseen--Frank energy functional $W$ for the
equilibrium configuration of a unit director field: \bea
W&=&\frac{k_1}{2}(\nabla\cdot d)^2+\frac{k_2}{2}|d\times (\nabla
\times d)|^2+\frac{k_3}{2}|d\cdot(\nabla\times d)|^2\non\\
&& +(k_2+k_4)[{\rm
tr} (\nabla d)^2-(\nabla \cdot d)^2].\label{OF}
\eea
We note that the forth term in \eqref{OF}
 $${\rm tr} (\nabla d)^2-(\nabla \cdot d)^2=\nabla
\cdot[(\nabla d) d-(\nabla \cdot d)d]$$
  is a null Lagrangian, which  only
depends on the value of the trace of $d$ on the boundary (cf. \cite{AG97}).

The kinematic transport of the director $d$ (denoted by $g$) is given by:
\begin{equation}
g_{i}=\lambda_1N_{i}+\lambda_2d_{j}A_{ji}=\lambda_1\left(
N_i+\frac{\lambda_2}{\lambda_1}d_jA_{ji}\right),\label{kitr}
\end{equation}
which represents the effect of macroscopic flow field on the microscopic
structure. The material coefficients $\lambda_1$ and $\lambda_2$ reflects the
molecular shape (Jeffrey's orbit \cite{Je})  and the slippery between the fluid and
the particles (see discussions in Section \ref{EnVarA}). The first term of \eqref{kitr} represents
the rigid rotation of the molecule, while the second term stands for the
stretching of the molecule by the flow.

The stress tensor $\sigma$ has the following form:
\bea
{\sigma}_{ij}&=&\mu_1d_{k}A_{kp}d_{p}d_{i}d_{j}+\mu_2N_{i}d_{j}+\mu_3d_{i}N_{j}+
\mu_4A_{ij}\non\\
&& +\mu_5A_{ik}d_{k}d_{j}+\mu_6d_{i}A_{jk}d_{k}. \label{v5}
\eea
These (independent) coefficients $\mu_1,...,\mu_6$, which may depend
on material and temperature, are usually called Leslie coefficients. These
coefficients are related to certain local correlations in the fluid
(cf. \cite{dG}). For convenience, $\mu_i's$ are called viscous
coefficients in later sections.

In order to handle the higher-order nonlinearities due to the
nonlinear constraint $|d|=1$ (i.e., the Lagrangian multipliers
$\beta$, $\gamma$), one can introduce a penalty (or relaxation)
approximation of Ginzbug--Landau type, by adding one term
$$\mathcal{F}(d)=\frac{1}{4\varepsilon^2}(|d|^2-1 )^2$$ in $W$.
Physically this term can be attributed to the extensibility of the
molecules. After the discussions for each $\varepsilon>0$, we then
take the limit as $\varepsilon \rightarrow 0$. This method is
motivated by the work on the gradient flow of harmonic maps into the
sphere (see, e.g., \cite{CS89}), but whether the solution of the
 Ericksen--Leslie system with Ginzburg--Landau approximation converges
 to that of the original one with constraint $|d|=1$ as $\varepsilon$ tends
 to zero is still a challenging problem. Nevertheless, the reformulated system with
penalty approximation also has natural physical interpretations. It
is similar to what Leslie proposed in \cite{Le79} for the flow of an
anisotropic liquid with varying director length. Mathematically, it
can also be related to models   for nematic liquid crystals with
variable degree of orientation proposed by Ericksen in \cite{E87}
under  specific conditions.  In particular, $\{x: d(x, t) =0\}$
represents the transition region of isotropic fluids. We refer to
\cite{LL95} for more discussions.

For  simplicity, in this paper, we focus on the
relaxation form of the elastic energy associated with $d$:
 \be
W(d)=\frac12|\nabla d|^2+\frac{1}{4\varepsilon^2}(|d|^2-1)^2.
 \label{WWW}
 \ee
 It is obvious that this choice of $W$ corresponds to the elastically isotropic
situation, i.e., taking $k_1=k_2=k_3=1$, $k_4=0$ in \eqref{OF}. The corresponding problem
with general Oseen--Frank energy \eqref{OF} can be treated in a
similar way, but the argument is more involved. With the choice of
the penalized energy $W$, we can remove the Lagrangian multipliers
and set $\gamma=\beta_j=0$. Since the inertial constant $\rho_1$ is
usually very small (cf. \cite{C74a}), we take $\rho_1=0$. Moreover,
we assume that the density is constant (which in turn yields the
incompressibility $\nabla \cdot v=0$) and the external forces
vanish, i.e., $\rho=1,\  F=0, \ G=0$ (cf. \cite{LL01}). Note that in
the incompressible cases the assumption $F=0$ means that there is no
exterior nonconservative forces.

Thus, the full Ericksen--Leslie system \eqref{the most primitive equ 1}--\eqref{the most
primitive equ 3} can be reformulated to
 \begin{eqnarray}
 v_t+v\cdot\nabla v+\nabla P&=&-\nabla\cdot(\nabla d
\odot\nabla d) +\nabla\cdot\sigma,
 \label{e1}\\
\nabla\cdot v&=&0, \label{e2}\\
d_t+(v\cdot\nabla )d-\Omega
d+\frac{\lambda_2}{\lambda_1}Ad&=&-\frac{1}{\lambda_1}\left(\Delta
d-f(d)\right), \label{e3}
\end{eqnarray}
where $$f(d)= \mathcal{F}'(d)=\frac{1}{\varepsilon^2}(|d|^2-1)d$$ and
$\sigma$ is given by \eqref{v5}. We denote by $\nabla d\odot \nabla
d$  the $3\times 3$-matrix whose $(i,j)$-entry is
$\nabla_i d\cdot \nabla_j d$, $1\leq i,j\leq 3$. In the following
text, we just set $\varepsilon=1$ and our results indeed hold for
any arbitrary but fixed $\varepsilon>0$.
In this paper, we will focus on the bulk
properties of the Ericksen--Leslie system. For this, we consider the equations
\eqref{e1}--\eqref{e3} subject to periodic boundary conditions
(i.e., in torus $\mathbb{T}^3$):
 \be v(x+e_i,t) = v(x,t), \ \
d(x+e_i,t)=d(x,t), \ \ \mbox{for} \ (x,t) \in
\partial Q\times\mathbb{R}^+,
  \label{B.C. in nonzero case}
  \ee
and initial conditions
 \be v|_{t=0}=v_0(x), \ \ \mbox{with} \ \ \nabla\cdot
v_0=0, \ \ \ d|_{t=0}=d_0(x), \ \ \mbox{for} \ x \in Q,
 \label{I.C. in nonzero case}
 \ee
 where $Q$ is a unit square in $\mathbb{R}^3$.

Due to temperature dependence of the Leslie coefficients, there
exists different behavior between  various coefficients (cf.
\cite{dG}): $\mu_4$-which does not involve the alignment
properties-is a rather smooth function of temperature; but all the other
$\mu's$ describe couplings between the molecule orientation and the flow, and are
thus affected by a decrease in the nematic order $|d|$. In this
paper, we just look at the isothermal case where $\mu's$ are assumed to
be constants. The following relations are frequently introduced in the literature (cf. \cite{Le68,Le79})
\begin{eqnarray}
 &&\lambda_1=\mu_2-\mu_3, \ \ \ \lambda_2=\mu_5-\mu_6, \label{lama1}\\
&&\mu_2+\mu_3=\mu_6-\mu_5. \label{lam2}
\end{eqnarray}
Relations given in \eqref{lama1} are necessary conditions in order to satisfy the equation of
motion identically (cf. \cite[Section 6]{Le68}). \eqref{lam2} is
called \emph{Parodi's relation} (cf. \cite{P70}), which is derived
from Onsager reciprocal relations expressing the equality of certain
relations between flows and forces in thermodynamic systems out of
equilibrium (cf. \cite{O31}). Under the assumption of Parodi's
relation, we see that the dynamics of an incompressible nematic
liquid crystal flow involve five independent Leslie coefficients in
\eqref{v5}.

Since the mathematical structure of the Ericksen--Leslie system is
quite complicated, past existing work was almost all restricted to
its simplified versions (cf. \cite{LL95, LL96, LLZ07, LSY07, CR,
SL08}). As far as the general Ericksen--Leslie system is concerned,
there is few known result in analysis (cf. e.g., \cite{LL01}). In
\cite{LL01}, well-posedness of the general Ericksen--Leslie system
\eqref{e1}--\eqref{e3} subject to Dirichlet boundary conditions  was
proved under the special assumption $\lambda_2=0$, which imposed an
extra constraint on those Leslie coefficients. This physically
indicates that the stretching due to the flow field is neglected,
which is more feasible for small molecules. Mathematically this
assumption brings great convenience since a weak maximum principle
for $|d|$ holds (cf. \cite[Theorem 3.1]{LL01}). For the general
system \eqref{e1}--\eqref{e3}, the maximum principle for $|d|$ fails
when $\lambda_2\neq 0$. This leads to extra difficulties in the
study of well-posedness, especially in dealing with those highly
nonlinear stress terms in $\sigma$ (cf. \cite{LL01,SL08}). Even in
the 2D case, it is hard to obtain global existence of solutions
without any further restriction on these viscous coefficients. This
is rather different from regular Newtonian fluid cases.

\textit{Summary of results}. The purpose of this paper is to study the connections between
physical parameters, namely, the Leslie coefficients and the
well-posedness as well as stability properties of the general
Ericksen--Leslie system \eqref{e1}--\eqref{e3}. In particular, we
focus on the role of Parodi's relation \eqref{lam2}.  Parodi's
relation is a consequence of Onsager's reciprocal relations in the
microscale descriptions of liquid crystals \cite{O31,O31-2}, which
are nevertheless independent of the second law of thermodynamics.
Although the physical interpretation of the reciprocal relation is
related to microscopic reversibility and laws of detailed balance of
microscopic dynamics \cite{Mazur97}, the thermodynamic basis of
Onsager's reciprocal relations have been discussed and debated by
many researchers (cf. \cite{Tr}). There are evidences that for
particular materials, Onsager's relations and their counterparts may
serve as stability conditions (cf. \cite{C74, Tr}). In this paper,
we provide specific mathematical verifications for the nematic
liquid crystal flow.

First, in Theorems \ref{global existence of
classical solution} and \ref{vconr}, under the assumption that the fluid
viscosity $\mu_4$ is sufficiently large, we show existence and uniqueness of global
solutions within suitable regularity classes and their long-time behavior (uniqueness of asymptotic limit). In this case, we see that
Parodi's relation is not a necessary assumption for the
well-posedness and long-time dynamics, while the large viscosity
constant $\mu_4$ plays a dominative role.

Next, without the largeness assumption on $\mu_4$, we first prove local wellposedness of the Ericksen--Leslie system
\eqref{e1}--\eqref{e3} (cf. Theorem \ref{locals solution}). Furthermore, in Theorem \ref{Main theorem III}, we prove global well-posedness and Lyapunov stability of the Ericksen--Leslie system, when the initial data is near certain equilibrium (local minimizer of the
elastic energy $W$ given by \eqref{WWW}). We see that Parodi's relation turns out to be crucial (as a sufficient condition) in obtaining well-posedness and (nonlinear) stability of the Ericksen--Leslie system.

Finally, we demonstrate the
connection between Parodi's relation and linear stability of the
original Ericksen--Leslie system \eqref{the most primitive equ 1}--\eqref{the
most primitive equ 3} (namely, with the constraint $|d|=1$). The result obtained in Theorem  \ref{proposition on currie} indicates that without Parodi's relation, the linearized Ericksen--Leslie system admits unstable plane wave solutions. In other words, Parodi's relation is a necessary condition for linear stability of the Ericksen--Leslie system.

\br
Our results presented in this paper are stated in the three dimensional case $n=3$.
When the spatial dimension $n=2$, if we consider the velocity field $v: \Omega\times [0,T]\to \mathbb{R}^2$
and director $d: \Omega\times [0,T]\to \mathbb{R}^2$, namely, the molecule director $d$ is also confined in
a plane, then it is easy to verify that all the results we obtained below for the 3D case also hold in 2D
(sometimes even under weaker assumptions, see e.g., Remark \ref{2Da}). However, if one wishes to consider
the Ericksen--Leslie system in a 2D domian $\Omega\subset \mathbb{R}^2$ but the director field $d$ is still
 allowed to be a three dimensional vector, some troubles will come up.
For instance, the parallel transport terms $\Omega d$ (rotation) and
$A d$ (stretching) cannot be properly defined, because $\Omega$ and
$A$ are $2\times 2$ matrices but $d$ is a 3D vector. We want to
mention that such problem does not apply to simplified liquid
crystal system of small molecules \cite{LL95}. In particular, we
refer to recent works \cite{LinLinWang10, LinWang10, XuZhang12} for
a simplified liquid crystal system in 2D but the director $d:
\Omega\times [0,T]\to S^2$, which is three dimensional (with the
constraint $|d|=1$). \er

\textit{Plan of the paper}. The remaining part of the paper is organized as follows. In Section
2, we discuss the basic energy dissipation law of the system
\eqref{e1}--\eqref{e3}.  In Section 3,  for the given energy law, we
re-derive the Ericksen--Leslie system via an energy variational approach. In
particular, under Parodi's relation, we show the specific relations between results from
Least Action Principle and those from Maximum Dissipation Principle. In Section 4, we prove global
well-posedness under large viscosity assumption on $\mu_4$ and the
long-time behavior of global solutions. In particular, we show that
any global solution will converge to a single steady state as time
tends to infinity and provide an estimate on the convergence rate.
In Section 5, we prove the well-posedness and stability when the
initial velocity is near zero and the initial director is close to a
local energy minimizer under Parodi's relation. In Section 6, we
discuss the connection between Parodi's relation and linear
stability of the original Ericksen--Leslie system.
In Section 7, the appendix
section, we present some detailed calculations needed for the previous
sections.


\section{Basic Energy Law}\label{basicEL}
\setcounter{equation}{0}

Generally speaking, singularities that can be observed for a
physical system are those energetically admissible ones (cf.
\cite{LLZ07}). It has been pointed out that the Ericksen--Leslie system
\eqref{e1}--\eqref{I.C. in nonzero case} obeys some dissipative
energy inequality under proper assumptions on those physical
coefficients (cf. \cite{LL01}).

The total energy of the Ericksen--Leslie system
\eqref{e1}--\eqref{I.C. in nonzero case} consists of kinetic and
potential energies and it is given by
 \be
 \mathcal{E}(t)=\frac{1}{2}\|v\|^2+\frac{1}{2}\|\nabla
d\|^2+\int_{Q}\mathcal{F}(d)dx.
 \label{total energy of the system}
 \ee
 For the sake of simplicity, we
denote the inner product on $L^2(Q)$ (or $\mathbf{L}^2(Q)$, for the
corresponding vector space) by $(\cdot,\cdot)$ and the associated
norm by $\|\cdot\|$.

By a direct calculation with smooth solutions
$(v,d)$ to the system \eqref{e1}--\eqref{I.C. in nonzero case}, we
have (cf. \cite[Theorem 2.1]{LL01} for detailed calculations for
the corresponding initial boundary value problem)
 \bea
  \frac{d}{dt}\mathcal{E}(t)&=&
-\int_{Q}\Big(\mu_1|d^{T}Ad|^2+\frac{\mu_4}{2}|\nabla
v|^2+(\mu_5+\mu_6)|Ad|^2
\Big)dx \non\\
&&+\lambda_1\|N\|^2+(\lambda_2-\mu_2-\mu_3)(N, Ad).
 \label{BELnoPa}
 \eea
Here
and after, we always assume that
 \bea
 \lambda_1&<&0, \label{lama1a}\\
 \mu_5+\mu_6 &\geq& 0,  \label{mu56}  \\
 \mu_1&\geq& 0, \quad \mu_4>0. \label{mu14}
 \eea
These assumptions are assumed to provide
necessary conditions for the dissipation of the director field \cite{E91, Le79}.
As indicated in \cite{LL01}, the assumption \eqref{lama1} guarantees the
existence of the Lyapunov-type functional. However, we note that Parodi's
relation \eqref{lam2} is not necessary in the derivation of
\eqref{BELnoPa}. If \eqref{lam2} is employed, i.e., $\lambda_2=-(\mu_2+\mu_3)$, we immediately arrive
at the energy inequality obtained in \cite[Theorem 2.1]{LL01}.
Moreover, if we further assume
$\lambda_2=0$, it follows from \eqref{BELnoPa}--\eqref{mu14} that
$\mathcal{E}(t)$ is decreasing in time, which is exactly the case
studied in \cite{LL01}.
 \bl [Basic energy law with Parodi's relation]
 \label{BEL}
 Suppose that the assumptions  \eqref{lama1},
 \eqref{lam2}, \eqref{lama1a}, \eqref{mu56} and \eqref{mu14} are satisfied. In addition, if we assume
 \be \frac{(\lambda_2)^2}{-\lambda_1} \leq
 \mu_5+\mu_6, \label{critical point of lambda 2}
 \ee
 then the total energy $\mathcal{E}(t)$ is decreasing in time such that
 \bea
 \frac{d}{dt}\mathcal{E}(t)
&=& -\int_{Q}\Big(\mu_1|d^{T}Ad|^2+\frac{\mu_4}{2}|\nabla v|^2
\Big)dx +\frac{1}{\lambda_1}\|\Delta d-f(d)\|^2
 \non\\
 &&-\Big(\mu_5+\mu_6+\frac{(\lambda_2)^2}{\lambda_1}\Big)\|Ad\|^2\non\\
 &\leq& 0.
 \label{basic energy law at the critical point of lambda 2}
 \eea
 \el
\begin{proof}
By Parodi's relation \eqref{lam2}, i.e., $\lambda_2=-(\mu_2+\mu_3)$,
we infer from the transport equation of $d$ (cf. \eqref{e3}) that
 \bea && \lambda_1\|N\|^2+(\lambda_2-\mu_2-\mu_3)(N, Ad)\non\\
 &
\stackrel{\eqref{lam2}}{=}&(N, \lambda_1N+\lambda_2Ad)+\lambda_2(N, Ad) \non\\
&=& (N, \lambda_1N+\lambda_2Ad)+ \left(N+\frac{\lambda_2}{\lambda_1}A d, \lambda_2 A d\right)-\frac{(\lambda_2)^2}{\lambda_1}\|Ad\|^2\non\\
&=& \lambda_1 \left\| N+\frac{\lambda_2}{\lambda_1}A d\right\|^2-\frac{(\lambda_2)^2}{\lambda_1}\|Ad\|^2\non\\
 &\stackrel{\eqref{e3}}{=}&
  \frac{1}{\lambda_1}\|\Delta
d-f(d)\|^2-\frac{(\lambda_2)^2}{\lambda_1}\|Ad\|^2. \label{chPa}
 \eea
 Inserting the above result into \eqref{BELnoPa}, we arrive at our
 conclusion. \qed
\end{proof}

On the contrary, if Parodi's relation \eqref{lam2} does not hold,
additional assumptions have to be imposed to ensure the dissipation of
the total energy.

\bl[Basic energy law without Parodi's relation] Suppose that \eqref{lama1}, \eqref{lama1a}, \eqref{mu56} and
\eqref{mu14} are satisfied. If we further assume that
 \be
|\lambda_2-\mu_2-\mu_3| \leq
2\sqrt{-\lambda_1}\sqrt{\mu_5+\mu_6},\label{noPa}
  \ee
then the following energy inequality holds:
 \be \frac{d}{dt}\mathcal{E}(t)\leq
-\int_{Q}\left(\mu_1|d^{T}Ad|^2+\frac{\mu_4}{2}|\nabla v|^2
\right)dx\leq 0.\label{belb}
 \ee
 Moreover, if
 \be
|\lambda_2-\mu_2-\mu_3| <
2\sqrt{-\lambda_1}\sqrt{\mu_5+\mu_6},\label{noPa1}
  \ee
then the dissipation in \eqref{belb} will be stronger in the sense
that there exists a small constant
  $\eta>0$,
  \bea
 \frac{d}{dt}\mathcal{E}(t)  &\leq&
-\int_{Q}\left(\mu_1|d^{T}Ad|^2+\frac{\mu_4}{2}|\nabla v|^2
\right)dx -\eta (\|Ad\|^2+\|N\|^2)\non\\
&\leq & 0.\label{belc}
 \eea
 \el
 \begin{proof}
 The conclusion easily follows from \eqref{BELnoPa} and the
 Cauchy--Schwarz inequality. \qed
 \end{proof}

\section{Energy Variational Approaches} \label{EnVarA}
\setcounter{equation}{0}
The energy variational approaches (\textit{EnVarA}) provide  unified
variational frameworks in studying  complex fluids with
microstructures  (cf. \cite {HKL}). From the energetic point of
view, the Ericksen--Leslie system \eqref{e1}--\eqref{e3} exhibits competition
between the macroscopic flow field and the microscopic director
field, through the coupling between the kinematic transport of the
director $d$ by the macroscopic velocity field $v$ and the averaged
microscopic effect in the form of induced macroscopic elastic stress
on the macroscopic flow field. This contributes to some interesting
hydrodynamic and rheological properties of the liquid crystal flows.
Based on the basic energy law in Section 2, and due to the special
feature of nematic liquid crystal flow such that the molecular
orientations are transported and deformed by the flow under parallel
transport, we shall develop a formal physical derivation of the
induced elastic stress through \emph{EnVarA}. This will provide us
with a further understanding of the competition between hydrodynamic
kinetic energy and internal elastic energy due to the presence of
the orientation field $d$.

The energetic variational treatment of complex fluids starts with
the energy dissipative law for the whole coupled system
\cite{HKL,ZGLWS}:
 \be
\dfrac{dE^{tot}}{dt}=-\mathcal{D},
 \non
 \ee
 where $E^{tot}=E^{kinetic}+E^{int}$ is the total
energy consisting of the kinetic energy and free energy. Here
$\mathcal{D}$ is the dissipation function which is equal to  the
entropy production of the system in isothermal situations. Following
Onsager's  linear response assumption, we assume that $\mathcal{D}$
is a linear combination of the squares of various rate functions
such as velocity, rate of strain or the material derivative of
internal variables (cf. \cite{O31,O31-2,O53,HKL,ZGLWS}). The
\emph{EnVarA} combines the maximum dissipation principle (for long
time dynamics) and the least action principle, or equivalently, the
principle of virtual work (for intrinsic and short time dynamics)
into a force balance law that expands the conservation law of
momentum to include dissipation (cf. \cite{La96, CH53}). The least
action principle gives us the Hamiltonian (reversible) part of the
system related to conservative forces. Meanwhile, the maximum
dissipation principle provides the dissipative (irreversible) part
of the system related to dissipative forces. In this way, we can
distinguish the conservative and dissipative parts among the induced
stress terms.

In the context of basic mechanics, both hydrodynamics and
elasticity,  the basic variable is the flow map $x(X, t)$ (particle
trajectory for any fixed $X$) . Here, $X$ is the original labeling
(the Lagrangian coordinate) of the particle, which is also referred
to as the material coordinate, while $x$ is the current (Eulerian)
coordinate and is also called the reference coordinate. For a given
velocity field $v(x, t)$, the flow map is defined by the ordinary
differential equations:
   \be
   x_t=v(x(X, t), t), \;\; x(X, 0)=X.\non
   \ee
The deformation tensor $\mathbb{F}$ associated with the flow field
is given by $$\mathbb{F}_{ij}=\frac{\partial x_i}{\partial X_j}.$$
Without ambiguity, we define $\mathbb{F}(x(X,t),t)=\mathbb{F}(X,t)$.
Applying the chain rule, we can see that $\mathbb{F}(x,t)$ and
$\mathbb{F}^{-T}(x,t)$ satisfy the following transport equations
(cf. e.g., \cite{GU, Lar})
\bea
&&  \mathbb{F}_t+v\cdot\nabla \mathbb{F}=\nabla v\mathbb{F},\non\\
&& \mathbb{F}^{-T}_t+v\cdot\nabla \mathbb{F}^{-T}=-\nabla^T
v\mathbb{F}^{-T}.\non
 \eea
 Kinematic transport of the director field $d$ represents the (microscopic) molecules moving in the (macroscopic) flow \cite{Lar, SL08}.
 For general ellipsoid shaped liquid crystal
molecules, the transport of $d$ can be represented by
 \be d(x(X,t),t)=\mathbb{E}d_0(X)\label{transE} \ee
 with
$d_0(X)$ being the initial configuration. The deformation tensor
$\mathbb{E}(x(X,t),t)$ carries all the information of micro
structures and configurations. It satisfies the following transport
equation whose right-hand side can also be reformulated into a
combination of a symmetric part and a skew part: (cf. \cite{SL08,
LLZ07,Je})
 \bea \mathbb{E}_t+v\cdot\nabla \mathbb{E}&=&\Big[\alpha\nabla v+(1-\alpha)(-\nabla^{T}v)
\Big]\mathbb{E}\non\\
&=& \Omega\mathbb{E}+(2\alpha-1)A\mathbb{E}.\label{mathcal
E equ 1}
 \eea
 Such solutions are called Jeffrey's orbits (cf. \cite{Je}). By the fundamental work of Jeffrey \cite{Je}, the parameter
  $$
  \eta=2\alpha-1=\frac{r^2-1}{r^2+1}\in [-1,1], \quad r\in \mathbb{R}
 $$
 is related to the aspect ratio of the ellipsoids. Recently, we have shown that $\eta$ can also be related to the slippage between the particles and the flow \cite{SL08}.
In the present case, we see that
$$\alpha=\frac12\left(1-\frac{\lambda_2}{\lambda_1}\right).$$

In what follows, we shall apply \emph{EnVarA} to recover the system
\eqref{e1}--\eqref{e3} from the basic energy law under the
assumption that both \eqref{lama1} and Parodi's relation
\eqref{lam2} are satisfied. The kinetic energy $E^{kinetic}$ and
internal elastic energy $E^{int}$ of the system
\eqref{e1}--\eqref{e3} are given by
\[ E^{kinetic}=\frac12\|v\|^2,\quad  E^{int}=E(d)=\frac{1}{2}\|\nabla d\|^2+\int_{Q}\mathcal{F}(d)dx. \]
 The Legendre transformation yields the action functional $ \mathbb{A}$ of the
particle trajectories in terms of the flow map $x(X,t)$:
 \be
 \mathbb{A}(x)=\int_0^T (E^{kinetic}-E^{int})dt,\non
 \ee
 which represents the competition between the kinetic energy and the internal energy.
If there is no internal microscopic damping, we deduce the total
(pure) transport equation of $d$  from \eqref{mathcal E equ 1} such that
 \bea \frac{Dd}{Dt}&=&d_t+v\cdot\nabla
d-\alpha\nabla v\, d+(1-\alpha)(\nabla^{T}v)d\non\\
&=& d_t+v\cdot \nabla
d-\Omega d+\frac{\lambda_2}{\lambda_1}Ad\non\\
&=& 0. \label{total transport
of d}
 \eea
The least action principle optimizes the action $\mathbb{A}$ with
respect to all volume preserving trajectories $x(X,t)$, i.e.,
 $\delta_x \mathbb{A}=0$, with incompressibility of the fluid $\nabla \cdot v=0$. Consequently,
we obtain
the conservative force
balance equation of classical Hamiltonian mechanics (see \eqref{weak LAP} for its weak variational form)
 \be v_t+v\cdot\nabla
v=-\nabla P-\nabla \cdot(\nabla d\odot\nabla
d)+\nabla\cdot\tilde{\sigma},\label{conseva part}
\ee
where
\bea
 \tilde{\sigma}&=& -
\frac{1}{2}\left(1-\frac{\lambda_2}{\lambda_1}\right) (\Delta
d-f(d))\otimes d\non\\
&&  + \frac{1}{2}\left(1+\frac{\lambda_2}{\lambda_1}\right)  d\otimes (\Delta
 d-f(d)).\label{tisigma}
 \eea
 Here, the symbol
$\otimes$ denotes the usual Kronecker multiplication, namely,
$(a\otimes b)_{i,j}=a_ib_j$ for $a,b\in \mathbb{R}^3$ and
$1\leq i, j\leq 3$. We also note that the stress tensor $\tilde{\sigma}$ is not symmetric due to the different coefficients of its two components. Together with \eqref{total transport
of d}, we recover the conservative (Hamiltonian) part of the full
system \eqref{e1}--\eqref{e3} (see Section \ref{cLAP} for the detailed
calculations).


On the other hand, taking the internal dissipation into account together with the transport
equation \eqref{total transport of d},
 we get
\bea d_t+v\cdot\nabla
d-\Omega\,d+\frac{\lambda_2}{\lambda_1}A\,d&=&\frac{1}{\lambda_1}\frac{\delta
E^{int}}{\delta d}\non\\
&=& -\frac{1}{\lambda_1}(\Delta d-f(d)),
 \label{transport equ of d derived from EVA}
  \eea
which reflects the elastic relaxation dynamics. The dissipation
functional $\mathcal{D}$ to the system \eqref{e1}--\eqref{e3} is in
terms of the variables $A$ and $N$ (cf. \eqref{BELnoPa}) (we remark
that our dissipation functional, like in \cite{B70}, departs from
those loosely defined by Onsager in \cite{O31-2}).
 Under Parodi's relation \eqref{lam2} and by \eqref{transport equ of d derived from EVA}, it can be transformed into the following form
(cf. \eqref{chPa})
 \bea \mathcal{D}&=&
\mu_1\|d^{T}Ad\|^2+\frac{\mu_4}{2}\|\nabla
v\|^2-\lambda_1\left\|d_t+v\cdot\nabla
d-\Omega\,d+\frac{\lambda_2}{\lambda_1}A\,d\right\|^2\non\\
&& +\Big(\mu_5+\mu_6+\frac{(\lambda_2)^2}{\lambda_1}\Big)\|Ad\|^2.\label{DIS2}
 \eea
 According to the maximum dissipation principle \cite{O31,O31-2,O53}, we take
$\delta_v\left(\frac12\mathcal{D}\right)=0$ (performing variation
with respect to the rate function, i.e., the velocity $v$ in
Eulerian coordinate) with incompressibility of the fluid $\nabla
\cdot v=0$. This yields the dissipative force balance law equivalent
to the conservation of momentum (see Section \ref{cMDP} for the
detailed calculations):
 \be 0=-\nabla P-\nabla\cdot(\nabla d\odot \nabla d)+
\nabla\cdot\sigma,\label{momentum equation from dissipative energy
a}
 \ee
 where
 \bea
  \sigma&= &\mu_1(d^TAd)d\otimes d+\mu_2N\otimes d
+\mu_3d\otimes N+\mu_4 A\non\\
&& +\mu_5 Ad\otimes d+\mu_6d\otimes
Ad,\label{sigma}
 \eea
with constants
 \be \mu_2=\frac{1}{2}(\lambda_1-\lambda_2), \ \ \
\mu_3=-\frac{1}{2}(\lambda_1+\lambda_2). \non
 \ee
Combining \eqref{momentum equation from dissipative energy a} with
\eqref{transport equ of d derived from EVA}, we recover the
dissipative part of the full system \eqref{e1}--\eqref{e3}, which
stands for the macroscopic long time dynamics.

The Ericksen--Leslie system \eqref{e1}--\eqref{e3} is the hybrid of these two conservative/dissipative systems.
Combining the dissipative part derived from maximum dissipation
principle (cf. \eqref{momentum equation from dissipative energy a})
with the conservative part derived from the least action principle
(cf. \eqref{conseva part}), and taking into account the total equation of $d$ \eqref{transport equ of d derived from
EVA}, we recover the full system
\eqref{e1}--\eqref{e3}.

 \begin{remark}
 We first observe from \eqref{tisigma} and \eqref{transport equ of d
derived from EVA} (i.e., $-\lambda_1N-\lambda_2A\,d=\Delta
d-f(d)$) that
 \be \tilde{\sigma} =\mu_2N\otimes d+\mu_3d\otimes
N+\eta_5A\,d\otimes d+\eta_6d\otimes A\,d,\label{vwrk}
 \ee
 with constants
 \bea && \mu_2=\frac{1}{2}(\lambda_1-\lambda_2), \ \
\mu_3=-\frac{1}{2}(\lambda_1+\lambda_2), \non\\
&&
\eta_5=\frac{1}{2}\left(\lambda_2-\frac{(\lambda_2)^2}{\lambda_1}\right),
\ \
\eta_6=-\frac{1}{2}\left(\lambda_2+\frac{(\lambda_2)^2}{\lambda_1}\right).
\label{eta 5 and eta 6 in POV}
 \eea
 The interesting fact from the above derivation is that the
induced stress terms
 \[-\nabla d \odot \nabla d+\mu_2 N\otimes  d+\mu_3 d \otimes  N
+\eta_5 Ad\otimes  d +\eta_6 d\otimes  Ad
 \]
 can be derived either by the least action principle (cf.
\eqref{vwrk}) or the maximum dissipation principle (contained in
\eqref{sigma}). Therefore, they can either be recognized as
conservative or dissipative. However, the remaining part in \eqref{sigma}
 \be \mu_1(d^T A d)d\otimes  d+\mu_4  A+(\mu_5-\eta_5)Ad\otimes  d
+(\mu_6-\eta_6)d\otimes  Ad
 \label{dissfor}
 \ee
 can only be derived by the
 maximum dissipation principle. This fact indicates that these terms
 in \eqref{dissfor}
 are
dissipative. In particular, at the critical value of $\lambda_2$,
i.e.,
 \be
 |\lambda_2|=\sqrt{-\lambda_1}\sqrt{\mu_5+\mu_6},\label{crilam2}
 \ee
 the dissipation functional of the
system \eqref{e1}--\eqref{e3} is reduced to
 \be
\mathcal{D}=\mu_1\|d^TAd\|^2+ \frac{\mu_4}{2}\|\nabla v\|^2
-\lambda_1\left\|d_t+v\cdot\nabla
d-\Omega\,d+\frac{\lambda_2}{\lambda_1}A\,d\right\|^2.\non
 \ee
 It turns out that $\mu_5=\eta_5$, $\mu_6=\eta_6$ and
the only two dissipative terms are given by those associated with
$\mu_1$ and $\mu_4$.\qed
\end{remark}

 Finally, we look at some special
cases of the system \eqref{e1}--\eqref{e3}. We assume that
\eqref{lama1}--\eqref{lam2} are satisfied and set
 \bea
 && \mu_1=0,\non\\
 && \mu_2=\frac{1}{2}\left(\lambda_1-\lambda_2\right), \qquad\quad
\mu_3=-\frac{1}{2}\left(\lambda_1+\lambda_2\right), \non\\
 && \mu_5=\frac{1}{2}\left(\lambda_2-\frac{(\lambda_2)^2}{\lambda_1}
\right), \quad
\mu_6=-\frac{1}{2}\left(\lambda_2+\frac{(\lambda_2)^2}{\lambda_1}
\right). \non
 \eea
 Since \eqref{crilam2} is now satisfied, then the system \eqref{e1}--\eqref{e3} can be reduced to
  \bea
&& v_t+v\cdot\nabla v+\nabla p=\frac{\mu_4}{2}\Delta v
-\nabla\cdot(\nabla d\odot\nabla
d) +\nabla\cdot \sigma, \label{1 for Huan's model}\\
&& \nabla\cdot v=0, \\
&& d_t+v\cdot\nabla d-\frac{\mu_2}{\lambda_1}\nabla
v\,d-\frac{\mu_3}{\lambda_1}\nabla^{T}v\,d=-\frac{1}{\lambda_1}(\Delta
d-f), \label{3 for Huan's model}
 \eea
where
 \be \sigma = -\frac{\mu_2}{\lambda_1}(\Delta d-f)\otimes  d
-\frac{\mu_3}{\lambda_1}d\otimes (\Delta d-f).
 \label{1a for Huan's model}
 \ee
  \begin{remark}\label{symp}
 The system \eqref{1 for Huan's model}--\eqref{1a for Huan's model} is consistent with these simplified models
 studied in \cite{GW3, LLZ07,SL08, LWX10}:

(1) The rod-like molecule model: $$ \mu_2=\lambda_1=-\lambda_2, \quad \mu_3=0.$$
 In this case, the director field $d$
satisfies the kinematic transport relation
$$d(x(X, t), t)=\mathbb{F}d_0(X), \quad \text{where}\quad \dot{\mathbb{F}} =\nabla v \mathbb{F}.$$

(2) The disc-like molecule model: $$ \mu_2=0, \quad \mu_3=-\lambda_1=-\lambda_2.$$
 In this case, $d$ satisfies
 $$ d(x(X,t), t)=\mathbb{F}^{-T} d_0(X),
  \quad \text{where}\quad \dot{\mathbb{F}}^{-T}=-\nabla^T v \mathbb{F}^{-T}.$$

(3) The sphere-like molecule model:
$$\mu_2=\frac{\lambda_1}{2}, \quad \mu_3=-\frac{\lambda_1}{2}, \quad
\lambda_2=0.$$
 In this case, $d$ satisfies $$ d(x(X, t),
t)=\mathbb{E} d_0(X),\quad \quad \text{where}\quad \dot{\mathbb{E}}=\frac{1}{2}(\nabla
v-\nabla^{T} v)\mathbb{E}.$$
\end{remark}

\section{Well-posedness and Long-time Behavior for Large Viscosity $\mu_4$}
\setcounter{equation}{0} \label{Larmu4}

For any Banach space $X$, we denote by $\mathbf{X}$ the vector space
$(X)^r$, $r\in \mathbb{N}$, endowed with the product norms.
 We recall the well established functional
settings for periodic problems (cf. \cite{Te}):
 \bea
H^m_p(Q)&=&\{u\in H^m(\mathbb{R}^3,\mathbb{R})\ |\ u(x+e_i)=u(x)\},\non\\
 \dot{H}^m_p(Q) &=& H^m_p(Q)\cap \left\{u:\ \int_Q u(x)dx=0\
\right\},\non\\
H&=&\{v\in \mathbf{L}^2_p(Q) ,\ \nabla\cdot v=0\},\ \ \text{where}\
\mathbf{L}^2_p(Q)=\mathbf{H}^0_p(Q),
\non\\
V&=&\{v\in \dot{\mathbf{H}}^1_p(Q),\ \nabla\cdot v=0\},\non\\
V'&=&\text{the\ dual space of\ } V.\non
 \eea
  We denote the inner
product on $L^2_p(Q)$ (or $\mathbf{L}^2_p(Q))$ as well as $H$ by
$(\cdot,\cdot)$ and the associated norm by $\|\cdot\|$. The space
$H^m_p(Q)$ will be short-handed by $H^m_p$ and the $H^m$-inner product
($m\in \mathbb{N}$) can be given by $ \langle v,
u\rangle_{H^m}=\sum_{|\kappa|=0}^m(D^\kappa v, D^\kappa u)$, where
$\kappa=(\kappa_1,..., \kappa_n)$ is a multi-index of length
$|\kappa|=\sum_{i=1}^n\kappa_i$ and
$D^\kappa=\partial_{x_1}^{\kappa_1},...,\partial_{x_n}^{\kappa_n}$.
We denote by $C$ the genetic constant possibly depending on
$\lambda_i's, \mu_i's, Q, f$ and the initial data. Special
dependence will be pointed out explicitly if necessary. Throughout
the paper, the Einstein summation convention will be used.

 As mentioned in the introduction, we use the Ginzburg--Landau approximation to reduce
 the order of nonlinearities caused by the constraint $|d|=1$. We note that
 either for a highly simplified liquid crystal model (cf. \cite{LL95}), or for the general Ericksen--Leslie system
\eqref{e1}--\eqref{I.C. in nonzero case} with the artificial
assumption $\lambda_2=0$ (cf. \cite{LL01}), a certain type of maximum principle holds
for the $d$-equation, namely, if $|d_0|\leq 1$ then $|d|\leq 1$.
This fact still holds for our current periodic settings with the
same assumption on $\lambda_2$. Then combing the basic energy law,
one can deduce that
 \bea && v \in L^{\infty}(0, T; H)\cap L^2(0, T; V), \label{weak reg1}\\
 && d \in L^{\infty}(0, T; \mathbf{H}^1_p\cap \mathbf{L}^\infty_p)\cap
 L^2(0, T; \mathbf{H}^2_p),\label{weak reg}
 \eea
 which is sufficient for the following formulation of weak
 solutions:
 \bd \label{weakD}
 $(v, d)$ is called a weak solution of \eqref{e1}--\eqref{I.C. in nonzero case} in
$Q_{T} = Q \times (0, T)$ if it satisfies \eqref{weak reg1}, \eqref{weak reg} and for
any smooth function $\psi(t)$ with $\psi(T)=0$ and $\phi(x) \in
\mathbf{H}^1_p$, the following weak formulation together with the initial and
boundary conditions \eqref{B.C. in nonzero case} and \eqref{I.C. in
nonzero case} hold:
 \bea
&&-\int_{0}^{T}(v, \psi_t\phi)dt+\int_{0}^{T}(v\cdot\nabla v,
\psi\phi)dt \non\\
=&& -(v_0, \phi)\psi(0)+\int_{0}^{T}(\nabla d\odot\nabla d,
\psi\nabla\phi)dt-\int_{0}^{T}(\sigma, \psi\nabla\phi)dt,\non
 \eea
where $\sigma$ is defined in \eqref{v1}, and
 \bea
&&-\int_{0}^{T}(d, \psi_t\phi)dt+\int_{0}^{T}(v\cdot\nabla d,
\psi\phi)dt-\int_{0}^{T}(\Omega d, \psi\phi)dt
+\frac{\lambda_2}{\lambda_1}\int_{0}^{T}(Ad, \psi\phi)dt
\non\\
&=&-(d_0, \phi)\psi(0)-\frac{1}{\lambda_1}\int_{0}^{T}(\Delta d -
f(d), \psi\phi)dt.\non
 \eea
 \ed
With the help of the maximum principle under the assumption $\la_2=0$, in \cite{LL01}, the authors obtained
the existence of weak solutions by applying a semi-Galerkin
procedure (cf. \cite{LL95} for the simplified liquid crystal system). For the more
general case considered in the present paper, we no longer assume
 that $\lambda_2 = 0$. Consequently, the kinetic transport
includes the stretching effect that leads to the loss of maximum
principle for $d$. In order to ensure that the extra
stress term $\nabla\cdot\sigma$ is well defined in the weak
formulation (cf. Definition \ref{weakD}), the regularity
$$ d \in L^\infty(0,T;\mathbf{L}^{\infty})$$
turns out to be essential (we
refer to \cite{SL08} for the discussions on the rod-like molecule
liquid crystal model, which is a special case of the general
system \eqref{e1}--\eqref{I.C. in nonzero case}). In the subsequent analysis, we have to confine ourselves to the
periodic boundary conditions, which helps us to avoid extra
difficulties involving boundary terms when performing integration by
parts in the derivation of higher-order energy inequalities.

Finally, we remark that existence of global weak solutions to simplified liquid crystal systems in Remark \ref{symp} has been obtained in \cite{CR} with a
suitable set of boundary conditions (i.e., homogeneous Dirichlet boundary condition for
 $v$ together with the homogeneous Neumann boundary condition for $d$). Their argument is based on an appropriate
choice of test functions that leads to a suitable weak formulation of the system and thus overcomes difficulties from the stretching effect. Quite recently, the existence of global weak solutions with energy bounds to the general Ericksen--Leslie system \eqref{e1}--\eqref{e3} has been proved in \cite{CRW} by extending the argument in \cite{CR}. Moreover, in \cite{CRW}, under Parodi's relation, the authors prove the local existence/uniqueness of classical solutions to the general Ericksen--Leslie system \eqref{e1}--\eqref{e3} and establish a Beale--Kato--Majda type blow-up criterion.

\subsection{Galerkin approximation}

Under periodic settings, one can define a mapping $S$ associated
with the Stokes problem:
 $S u=-\Delta u$ for $ u\in D(S)=\{u\in H, S u\in H\}=\dot{\mathbf{H}}^2_p\cap H$.
 The operator $S$ can be seen as an unbounded
positive linear self-adjoint operator on $H$. If $D(S)$ is endowed
with the norm induced by $\dot{\mathbf{H}}^0_p$, then $S$ becomes an
isomorphism from $D(S)$ onto $H$.

 Let $\{\phi_i\}_{i=1}^\infty$ with
$\|\phi_i\|=1$ be the eigenvectors of the Stokes operator $S$ in the
periodic case with zero mean,
 \be -\Delta\phi_i + \nabla P_i = \kappa_i\phi_i,\quad
\nabla\cdot\phi_i = 0 \ \ \mbox{in Q},\quad
\int_{Q}\phi_i(x)\,dx=0,\non
 \ee
 where $P_i \in L^{2}$ and $0<\kappa_1\leq\kappa_2\leq ...$ are eigenvalues. The eigenvectors
 ${\phi_i}$ are smooth and the sequence $\{\phi_i\}_{i=1}^\infty$ forms an orthogonal
basis of $H$ (cf. \cite{Te}). Let
$$ \mathrm{P}_m: H\rightarrow H_m
\doteq span\{\phi_1, \cdots, \phi_m\}, \quad m\in \mathbb{N}.
 $$
 We
consider the following (variational) approximate problem:
 \bea
 && (\partial_t v_m, u_m)+ (v_m\cdot\nabla
v_m, u_m) \non\\
&&\ \ =  (\nabla d_m\odot\nabla d_m, \nabla u_m)-
(\sigma_m,\nabla u_m), \quad \forall\  u_m\in H_m,\label{equation 1 in approximate system}\\
 && N_m+\frac{\lambda_2}{\lambda_1}A_m d_m=-\frac{1}{\lambda_1}\Delta
 d_m-f(d_m), \label{equation 3 in approximate system}\\
 &&v_m(x, 0)=\mathrm{P}_m v_0(x), \ \ d_m(x, 0)=d_0(x), \label{I.C. in approximate problem}\\
 &&v_m(x+e_i, t)=v_m(x, t), \ \ d_m(x+e_i, t)=d_m(x,t),
 \label{B.C. in approximate problem}
 \eea
where
 \bea
 \Omega_{m} &=&\frac12(\nabla v_{m}-\nabla^{T}v_{m}), \ \
 A_{m}=\frac12(\nabla v_{m}+\nabla^{T}v_{m}),\non\\
  N_m&=&\partial_t d_{m}+ (v_m\cdot\nabla) d_m+\Omega_m
 d_m,\non
  \\
\sigma_m&=&\mu_1(d_{m}^{T}A_{m}d_{m})d_{m}\otimes
d_{m}+\mu_2N_{m}\otimes d_{m}+\mu_3d_{m}\otimes
N_{m}+\mu_4A_{m}
\non\\
&&+\mu_5A_{m}d_{m}\otimes d_{m}+\mu_6d_{m}\otimes A_{m}d_{m}.\non
 \eea
We can prove local well-posedness of the approximate problem
\eqref{equation 1 in approximate system}--\eqref{B.C. in approximate
problem} by a similar semi-Galerkin procedure like \cite{SL08} (see
also \ \cite{LL95,LL01}). Smoothness of the approximate solutions in
the interior of $Q_{T_0}=(0,T_0)\times Q$ follows from the regularity theory for
parabolic equations and a bootstrap argument (cf. \cite{LA1,LL95}).
The uniqueness of smooth solutions can be proved by performing energy estimates on the difference of two
different solutions and using Gronwall's inequality. Since the proof is
standard, we omit the details here.
 \bp
 \label{Theorem for existence of approximate problem}
 Suppose that  $v_0 \in V$, $d_0 \in \mathbf{H}^2_p$. For any $m > 0$, there is a $T_0
> 0$ depending on $v_0$, $d_0$ and $m$ such that the
approximate problem \eqref{equation 1 in approximate
system}--\eqref{B.C. in approximate problem} admits a unique weak solution
$(v_m, d_m)$ satisfying
 \bea
 &&
 v_m \in L^{\infty}(0, T_0; V) \cap L^2(0, T_0; \mathbf{H}^2_p),\non\\
 && d_m \in L^{\infty}(0, T_0; \mathbf{H}^2_p)
\cap L^2(0, T_0;  \mathbf{H}^3_p).\non
 \eea
 Furthermore, $(v_m,
d_m)$ is smooth in the interior of $Q_{T_0}$.
 \ep

\subsection{Uniform \emph{a priori} estimates}

In order to prove global existence of solutions to the problem
\eqref{e1}--\eqref{I.C. in nonzero case}, we need some uniform
estimates that are independent of the approximate parameter $m$ and
the local existence time $T_0$. These uniform estimates enable us
(i) to pass to the limit as $m\to \infty$ to obtain a weak solution
to the system \eqref{e1}--\eqref{I.C. in nonzero case} in proper
spaces; (ii) to extend the local solution to a global one on $[0,
+\infty)$. One advantage of the above mentioned semi-Galerkin scheme
is that the approximate solutions satisfy the same basic energy law
and higher-order differential inequalities as the smooth solutions
to the system \eqref{e1}--\eqref{I.C. in nonzero case}. For the sake
of simplicity, the following calculations are carried out formally
for smooth solutions. They can be justified by using the approximate
solutions to \eqref{equation 1 in approximate system}--\eqref{B.C.
in approximate problem} and then passing to the limit.

The basic energy law plays an important role in the derivation of
uniform estimates on $\mathbf{L}^2\times \mathbf{H}^1$-norm of
$(v,d)$. According to the discussions in Section \ref{basicEL}, we
consider the following two cases, in which the basic energy law
holds:\medskip
 \begin{itemize}
 \item
\textbf{Case I} (with Parodi's relation): Suppose $\lambda_2\neq 0$,
\eqref{lama1},
 \eqref{lam2}, \eqref{lama1a}--\eqref{critical
point of lambda 2};
\medskip
 \item \textbf{Case II} (without Parodi's relation): Suppose
 $\lambda_2\neq 0$, \eqref{lama1}, \eqref{lama1a}--\eqref{mu14} and \eqref{noPa1}.
 \end{itemize}
\medskip
First, we consider \textbf{Case I}. It follows from Lemma \ref{BEL}
that
 \be
\frac{d}{dt}\mathcal{E}(t) \leq -\int_{Q}
\mu_1|d^TAd|^2dx -\frac{\mu_4}{2}\|\nabla v\|^2
+\frac{1}{\lambda_1}\|\Delta d-f(d)\|^2,\quad \forall\, t\geq 0.\non
 \ee
  This easily implies the following uniform estimates
 \bea
 &&\|v(t, \cdot)\|\leq C, \quad \|d(t, \cdot)\|_{\mathbf{H}^1}\leq C, \quad \forall\, t\geq
 0,\label{low1}
 \\
 && \int_0^{+\infty} \left(\int_{Q}\mu_1|d^TAd|^2 dx +\frac{\mu_4}{2}\|\nabla
 v\|^2-\frac{1}{\lambda_1}\|\Delta d-f(d)\|^2\right) dt\leq
 C,\label{low2}
 \eea
 where the constant $C>0$ depends only on $\|v_0\|$ and
 $\|d_0\|_{\mathbf{H}^1}$.

As we have mentioned before, the regularity $d \in
L^\infty(0,T;\mathbf{L}^{\infty})$ is crucial to ensure that the
extra stress term $\nabla \cdot \sigma$ can be suitably defined in
the weak formulation. Due to the lack of maximum principle for $d$,
an alternative way is to prove higher-order (spatial) regularity of $d$, e.g., in
$L^\infty(0,T; \mathbf{H}^2)$ and use the Sobolev embedding
$\mathbf{H}^2\hookrightarrow \mathbf{L}^\infty$ ($n= 3$). For this
purpose, we derive a new type of higher-order energy inequality,
which turns out to be useful in the study of global existence of
 solutions as well as the long-time behavior (cf. \cite{LL95, LL01, SL08, LW08} for simplified liquid
crystal systems).
 \bl
 \label{higher order energy law in large viscosity case}
  Set
 \be \mathcal{A}(t)=\|\nabla v(t)\|^2+\|\Delta
d(t)-f(d(t))\|^2.
 \label{higher energy sequence}
 \ee
 Let $\underline{\mu}$ be an arbitrary positive constant. We suppose that
 $\mu_4\geq \underline{\mu}$. For $n=3$, under the assumption of \textbf{Case I}, the following inequality holds:
 \bea  \frac{d}{dt}\mathcal{A}(t) &\leq&
-\left(\frac{\mu_4}{2}-C_1{\mu_4}^{\frac{1}{2}}\tilde{\mathcal{A}}(t)
\right)\|\Delta v\|^2\non\\
&&
 +\Big(\frac{1}{2\lambda_1}+C_2\mu_4^{-\frac{1}{4}}\tilde{\mathcal{A}}(t)
 \Big)
 \|\nabla(\Delta d-f)\|^2+C_3\mathcal{A}(t),
 \label{higher order energy inequality in large viscosity case}
 \eea
 where $$\tilde{\mathcal{A}}(t)=\mathcal{A}(t)+1,$$
  $C_i$ ($i=1,2,3$) are constants depending on $Q$, $f$, $\|v_0\|$,
 $\|d_0\|_{\mathbf{H}^1}$, $\lambda_1$, $\lambda_2$, $\mu_i$   $(i=1,2,3,5,6)$ and $\underline{\mu}$.
 \el
\begin{proof}
Without loss of generality, we assume that $\underline{\mu}= 1$. The
argument is valid for arbitrary but fixed $\underline{\mu}>0$.

Using
\eqref{e1}--\eqref{e3} and integration by parts, due to the periodic
boundary conditions, we obtain that (see Section \ref{cdA} for detailed
computations)
 \bea
&&\frac{1}{2}\frac{d}{dt}\mathcal{A}(t)+\mu_1\int_{Q}(d_kd_p\nabla_lA_{kp})^2dx+\frac{\mu_4}{2}\|\Delta
v\|^2\non\\
&& \ \ +(\mu_5+\mu_6)\int_{Q}|d_j\nabla_lA_{ji}|^2dx-\frac{1}{\lambda_1}\|\nabla(\Delta d-f)\|^2  \non\\
&=&-\mu_1\int_{Q}A_{kp}\nabla_l(d_kd_p) d_id_j\nabla_lA_{ij}dx
-\mu_1\int_{Q}A_{kp}d_kd_p\nabla_l(d_id_j)\nabla_lA_{ij}dx \non\\
&&-(\mu_5+\mu_6)\int_{Q}\nabla_ld_j d_kA_{ki} \nabla_lA_{ij}dx
-(\mu_5+\mu_6)\int_{Q}d_j\nabla_ld_kA_{ki} \nabla_lA_{ij}dx
\non\\
&&-\int_Q\nabla_l (\Delta d_i-f_i) \Omega_{ij}\nabla_l d_j
dx+\int_Q(\Delta d_i-f_i) \nabla_l\Omega_{ij}\nabla_l d_j dx\non\\
&&
+2\lambda_2\int_Q N_i\nabla_l
A_{ij}\nabla_l d_j dx+\lambda_2(N, A\Delta d)
\non\\
&&
-\frac{(\lambda_2)^2}{\lambda_1}\int_{Q}|\nabla_l (A_{ij}d_j)|^2dx
+(\Delta v, v\cdot\nabla v)+\frac{1}{\lambda_1}\int_{Q}f'(d)|\Delta d-f|^2dx \non\\
&& -\Big(\Delta d-f, f'(d)\big( \Omega d
-\frac{\lambda_2}{\lambda_1}A\,d\big)\Big)+2\int_Q\nabla_j (\Delta
d_i-f_i)\nabla_l
v_j\nabla_ld_idx\non\\
&& -(\Delta d-f, v\cdot\nabla f)\non\\
&\triangleq& I_1+\ldots+I_{14}.
 \label{second time expansion of derivative of A(t)}
\eea
In what follows, we estimate the right-hand side of \eqref{second time
expansion of derivative of A(t)} term by term.
 \be
I_1 \leq \frac{\mu_1}{4}\int_{Q}(d_kd_p\nabla_lA_{kp})^2dx+
C\|d\|_{\mathbf{L}^\infty}^2\|\nabla v\|^2_{\mathbf{L}^3}\|\nabla
d\|_{\mathbf{L}^6}^2.\non
 \ee
By the estimate \eqref{low1}, we infer from the Agmon's inequality
that
 \be \|d\|_{\mathbf{L}^\infty}\leq C(1+\|\Delta
 d\|^\frac12).\label{Agmon}
 \ee
Then from  \eqref{low1}, \eqref{Agmon} and the Gagliardo--Nirenberg
inequality, we obtain
 \be  \|\nabla v\|_{\mathbf{L}^3} \leq \|\nabla
v\|^\frac12\|\Delta v
  \|^\frac12, \ \
  \|\nabla v\|_{\mathbf{L}^4} \leq \|\nabla v\|^\frac14\|\Delta v
  \|^\frac34, \ee
  \be
   \|\nabla d\|_{\mathbf{L}^6} \leq C(\|\Delta d\|+1), \label{how to control nabla delta d 1}
   \ee
   \be
 \|\Delta d\| \leq \|\Delta d-f(d)\|+\|f(d)\|\leq \|\Delta
 d-f(d)\|+C, \ee
 \bea
  \|\nabla \Delta d\| &\leq& \|\nabla(\Delta
d-f(d)\|+\|\nabla
 f(d)\|\non\\
 & \leq & \|\nabla(\Delta d-f(d)\|+\|f'(d)\|_{\mathbf{L}^\infty}\|\nabla
 d\|\non\\
 &\leq& \|\nabla(\Delta d-f(d)\|+ C(1+\|d\|^2_{\mathbf{L}^\infty})
 \non\\
 &\leq &
  \|\nabla(\Delta d-f(d)\|+ C(1+\|\Delta d\|)
 \non\\
 &\leq& \|\nabla(\Delta d-f(d)\|+C(1+ \|\nabla \Delta
 d\|^\frac12\|\nabla d\|^\frac12+\|\nabla d\|)\non
 \\
 &\leq & \|\nabla(\Delta d-f(d)\| + \frac12 \|\nabla \Delta
 d\|+C. \label{how to control nabla delta d 2}
 \eea
 As a result, it holds
 \bea && \|d\|_{\mathbf{L}^\infty}^2\|\nabla v\|^2_{\mathbf{L}^3}\|\nabla d\|_{\mathbf{L}^6}^2
\non\\
&\leq &
  C\|\nabla v\|\|\Delta v\|(\|\Delta d-f\|^3+1)\non\\
&\leq& \left(\mu_4^{\frac12}+ \mu_4^\frac12 \|\Delta d-f\|^2\right)
\|\Delta v\|^2 + C\mu_4^{-\frac12}\|\nabla v\|^2(1+\|\Delta d-f\|^4) \non\\
 &\leq& \mu_4^{\frac12}\tilde{\mathcal{A}}\|\Delta v\|^2
 +C\mu_4^{-\frac12}\|\nabla
 v\|^2\non\\
 && \ \ +C\mu_4^{-\frac12}\|\nabla
 v\|^2\left(\|\nabla \Delta d\|^\frac12\|\nabla d\|^\frac12+\|\nabla
 d\|+C\right)^4 \non\\
 &\leq& \mu_4^{\frac12}\tilde{\mathcal{A}}\|\Delta v\|^2 +C\mu_4^{-\frac12}\|\nabla
 v\|^2
 +C\mu_4^{-\frac12}\|\nabla
 v\|^2 (\|\nabla (\Delta d-f)\|^2+1)\non\\
 &\leq& \mu_4^{\frac12}\tilde{\mathcal{A}}\|\Delta v\|^2
 +C\mu_4^{-\frac12}\mathcal{A}\|\nabla (\Delta d-f)\|^2 +
 C\mathcal{A},
  \label{6th power of nabla d}
  \eea
which implies that
 \bea I_1 &\leq&  \frac{\mu_1}{4}\int_{Q}(d_id_j\nabla_lA_{ij})^2dx
 +\mu_4^{\frac12}\tilde{\mathcal{A}}\|\Delta v\|^2\non\\
 &&
 +C\mu_4^{-\frac12}\mathcal{A}\|\nabla (\Delta d-f)\|^2 +
 C\mathcal{A}.
 \eea
 For $I_2$, using integration by parts, we obtain
 \bea
I_2&=& \mu_1\int_{Q}\nabla_lA_{kp}d_kd_p \nabla_l(d_id_j)
A_{ij}dx+\mu_1\int_{Q} A_{kp}\nabla_l(d_kd_p)  \nabla_l(d_id_j) A_{ij} dx \non\\
&&+\mu_1\int_{Q} A_{kp}d_k d_p (d_j\Delta d_i +2\nabla_l d_i\nabla_l d_j+d_i\Delta d_j) A_{ij}dx
\non\\
&\leq& \frac{\mu_1}{4}\int_{Q}(d_kd_p\nabla_lA_{kp})^2dx+
C\|d\|_{\mathbf{L}^\infty}^2\|\nabla v\|^2_{\mathbf{L}^3}\|\nabla
d\|_{\mathbf{L}^6}^2\non\\
 &&\ \ +C\|\nabla v\|_{\mathbf{L}^4}^2\|\Delta
d\|\|d\|_{\mathbf{L}^{\infty}}^3, \label{mu1ES2}
 \eea
where
\bea &&  C\|\nabla v\|_{\mathbf{L}^4}^2\|\Delta
d\|\|d\|_{\mathbf{L}^{\infty}}^3 \non\\
 & \leq&
 C \|\nabla
v\|^{\frac{1}{2}} \|\Delta
v\|^{\frac{3}{2}}(\|\Delta d-f\|^\frac52+1)\non\\
 &\leq& {\mu_4}^{\frac{1}{2}}\tilde{\mathcal{A}}\|\Delta
v\|^2+C{\mu_4}^{-\frac{3}{2}}\|\nabla v\|^2(1+\|\Delta
d-f\|^4).\label{mmm1}
 \eea
Thus, the right-hand side of \eqref{mmm1} can be estimated exactly
as \eqref{6th power of nabla d}. Therefore,
\bea
  I_2
&\leq & \frac{\mu_1}{4}\int_{Q}(d_kd_p\nabla_lA_{kp})^2dx
 +\mu_4^\frac12\tilde{\mathcal{A}}\|\Delta v\|^2\non\\
 &&\ \
 +C \mu_4^{-\frac12}\mathcal{A}\|\nabla (\Delta
 d-f)\|^2+C\mathcal{A}.
\eea
Concerning $I_3$ and $I_4$, we deduce from \eqref{how to control
nabla delta d 2} that (using again \eqref{6th power of nabla
d} and $\mu_4\geq 1$)
 \bea
I_3+I_4&=&-(\mu_5+\mu_6)\int_{Q}\nabla_ld_jd_k A_{ki}
\nabla_lA_{ij}dx \non\\
&& \ \ -(\mu_5+\mu_6)\int_{Q}d_j \nabla_ld_k A_{ki}
\nabla_lA_{ij}dx
\non\\
&\leq& C\|\Delta v\|\|\nabla v\|_{\mathbf{L}^3}\|\nabla d\|_{\mathbf{L}^6}\|d\|_{\mathbf{L}^\infty}\non\\
&\leq& {\mu_4}^{\frac{1}{2}}\|\Delta
v\|^2+ C{\mu_4}^{-\frac{1}{2}}\|\nabla v\|_{\mathbf{L}^3}^2\|\nabla d\|_{\mathbf{L}^6}^2\|d\|_{\mathbf{L}^\infty}^2\non\\
&\leq& {\mu_4}^{\frac{1}{2}}\|\Delta
v\|^2+ C\|\nabla v\|_{\mathbf{L}^3}^2\|\nabla d\|_{\mathbf{L}^6}^2\|d\|_{\mathbf{L}^\infty}^2\non\\
&\leq& 2 \mu_4^{\frac12}\tilde{\mathcal{A}}\|\Delta v\|^2
 +C\mu_4^{-\frac12}\mathcal{A}\|\nabla (\Delta d-f)\|^2 +
 C\mathcal{A}.\label{mmm56}
 \eea
Next,  for $I_5$, $I_6$, we have
 \bea
 I_5&\leq& C\|\nabla(\Delta d-f)\|\|\nabla v\|_{\mathbf{L}^3}\|\nabla
d\|_{\mathbf{L}^6}\non\\
&\leq& C\|\nabla(\Delta d-f)\|\|\nabla
v\|^\frac12\|\Delta v\|^\frac12(\|\Delta d-f\|+1)
\non\\
&\leq& \mu_4^{\frac{1}{4}}\|\nabla v\|\|\Delta
v\|+C\mu_4^{-\frac{1}{4}}\tilde{\mathcal{A}}\|\nabla(\Delta
d-f)\|^2    \non\\
&\leq& \mu_4^{\frac{1}{2}}\|\Delta
v\|^2+C\mu_4^{-\frac{1}{4}}\tilde{\mathcal{A}}\|\nabla(\Delta
d-f)\|^2+C\|\nabla v\|^2,\label{EsOM1}
\eea
and
 \bea
I_6&\leq& \|\nabla\Omega\|\|\Delta d-f\|\|\nabla
d\|_{\mathbf{L}^\infty}\non\\
& \leq &  C\|\Delta v\|\|\Delta
d-f\|(\|\nabla(\Delta
d-f)\|^{\frac{3}{4}}+1) \non\\
&\leq& \mu_4^\frac12\|\Delta v\|^2+ C\mu_4^{-\frac23} \|\Delta
d-f\|^2 \|\nabla(\Delta d-f)\|^2+C\|\Delta
d-f\|^2\non\\
 &\leq& \mu_4^{\frac{1}{2}}\|\Delta
v\|^2+C\mu_4^{-\frac23} \mathcal{A}\|\nabla(\Delta
d-f)\|^2+C\mathcal{A}. \label{EsOM2}
 \eea
Using integration by parts and \eqref{e3}, we get
 \bea && I_7+I_8+I_9\non\\
 &=&2\lambda_2\int_Q N_i\nabla_l
A_{ij}\nabla_l d_j dx+\lambda_2(N, A\Delta d)
-\frac{(\lambda_2)^2}{\lambda_1}\int_{Q}|\nabla_l (A_{ij}d_j)|^2dx
\non\\
&=& \lambda_2\int_Q N_i\nabla_l A_{ij}\nabla_l d_j
dx-\lambda_2\int_Q \nabla_lN_i A_{ij}\nabla_l d_j dx\non\\
&& \ \
-\frac{(\lambda_2)^2}{\lambda_1}\int_{Q}|A_{ij}\nabla_l d_j|^2dx
 -\frac{2(\lambda_2)^2}{\lambda_1}\int_{Q}\nabla_l A_{ij} d_j
A_{ik}\nabla_l d_k
dx  \non\\
&&\ \
-\frac{(\lambda_2)^2}{\lambda_1}\int_{Q}|\nabla_l A_{ij} d_j|^2dx
\non\\
&=& -\frac{\lambda_2}{\lambda_1}\int_Q (\Delta d_i-f_i)\nabla_l
A_{ij}\nabla_l d_j dx+\frac{\lambda_2}{\lambda_1}\int_Q
\nabla_l(\Delta d_i-f_i) A_{ij}\nabla_l d_j dx
 \non\\
&& -\frac{2(\lambda_2)^2}{\lambda_1}\int_{Q}\nabla_l A_{ij} d_j
A_{ij}\nabla_l d_j
dx-\frac{(\lambda_2)^2}{\lambda_1}\int_{Q}|\nabla_l A_{ij} d_j|^2dx
\non\\
&:=& K_1+K_2+K_3+K_4.\label{NAd1}
 \eea
Similar to \eqref{mmm56}, \eqref{EsOM1} and \eqref{EsOM2}, we have
 \bea &&
K_1+K_2 \non\\
&\leq& C\|\Delta v\|\|\Delta d-f\|\|\nabla
d\|_{\mathbf{L}^\infty}+C\|\nabla(\Delta d-f)\|\|\nabla
v\|_{\mathbf{L}^3}\|\nabla d\|_{\mathbf{L}^6}\non\\
&\leq& \mu_4^{\frac{1}{2}}\|\Delta v\|^2+C\mu_4^{-\frac14}
\mathcal{A}\|\nabla(\Delta d-f)\|^2+C\mathcal{A},\non
\eea
\be K_3 \leq {\mu_4}^{\frac{1}{2}}\tilde{\mathcal{A}}\|\Delta
v\|^2+C{\mu_4}^{-\frac{3}{2}}\|\nabla v\|^2.\non
 \ee
 Furthermore, \eqref{mu56} and \eqref{critical point
of lambda 2} indicate that
 \be
K_4 - (\mu_5+\mu_6)\int_{Q}|\nabla_l A_{ij} d_j|^2dx\leq 0.
\label{NAD}
 \ee
As a result,
 \bea
 && I_7+I_8+I_9\non\\
 & \leq & \mu_4^{\frac{1}{2}}\tilde{\mathcal{A}}\|\Delta
v\|^2+(\mu_5+\mu_6)\int_{Q}|\nabla_l A_{ij}
d_j|^2dx\non\\
&& \ \ +C\mu_4^{-\frac14} \mathcal{A}\|\nabla(\Delta
d-f)\|^2+C\mathcal{A}.\non
 \eea
For terms $I_{10}$, $I_{11}$ and $I_{12}$, we have
 \bea I_{10}& \leq& \|v\|_{\mathbf{L}^4}\|\nabla v\|_{\mathbf{L}^4}\|\Delta v\|
\non\\
& \leq & C\|v\|^{\frac{1}{4}}\|\nabla v\|^{\frac{3}{4}}\|\nabla
v\|^{\frac{1}{4}}\|\Delta v\|^{\frac{3}{4}}\|\Delta v\| \non\\
&\leq& {\mu_4}^{\frac{1}{2}}\|\Delta
v\|^2+{\mu_4}^{\frac{1}{2}}\|\nabla v\|^2\|\Delta
v\|^2+C{\mu_4}^{-\frac{7}{2}}\|\nabla v\|^2 \non\\
&\leq& {\mu_4}^{\frac{1}{2}}\tilde{\mathcal{A}}\|\Delta
v\|^2+C\mathcal{A},\label{VVV} \eea
 \bea
I_{11}& \leq&
C(\|d\|_{\mathbf{L}^6}^2+1)\|\Delta d-f\|_{\mathbf{L}^3}^2 \non\\
& \leq &
C\Big(\|\Delta d-f\|\|\nabla(\Delta
d-f)\|+\|\Delta d-f\|^2\Big)  \non\\
&\leq& -\frac{1}{4\lambda_1}\|\nabla(\Delta d-f)\|^2+C\|\Delta
d-f\|^2,
 \eea
 and
 \bea
I_{12}&=& -\Big(\Delta d-f, f'(d)\big(
\Omega\,d-\frac{\lambda_2}{\lambda_1}Ad\big) \Big)\non\\
&\leq&
C(\|d\|_{\mathbf{L}^6}^2+1)\|d\|_{\mathbf{L}^6}\|\Delta
d-f\|_{\mathbf{L}^3}\|\nabla
v\|_{\mathbf{L}^6} \non\\
 &\leq&C\Big(\|\nabla(\Delta d-f)\|+\|\Delta
d-f\|\Big)\|\nabla v\|^{\frac{1}{2}}
\|\Delta v\|^{\frac{1}{2}} \non\\
&\leq& \mu_4^{\frac{1}{2}}\|\Delta
v\|^2+{\mu_4}^{-\frac{1}{4}}\|\nabla(\Delta d-f)\|^2+C\mathcal{A}.
 \eea
  The estimate for $I_{13}$ is exactly the same as
\eqref{EsOM1} such that
 \be I_{13}
\leq {\mu_4}^{\frac{1}{2}}\|\Delta
v\|^2+C\mu_4^{-\frac{1}{4}}\tilde{\mathcal{A}}\|\nabla(\Delta
d-f)\|^2+C\|\nabla v\|^2.\label{Esfinal}
 \ee
 Finally, for $I_{14}$, we see that
 \bea
 I_{14}&\leq& C\|\Delta d-f\|_{\mathbf{L}^3}\|v\|_{\mathbf{L}^6}\|\nabla f\|
 \non\\
 &\leq & C\|\Delta d-f\|_{\mathbf{L}^3}\|\nabla v\|(1+\|d\|^2_{\mathbf{L}^\infty})\|\nabla
 d\|\non\\
 &\leq& C(1+\|\Delta d-f\|)(\|\nabla (\Delta d-f)\|+\|\Delta d-f\|)\|\nabla
 v\|\non\\
 &\leq& \mu_4^\frac12\mathcal{A}\|\Delta
 v\|^2+\left(-\frac{1}{4\lambda_1}+\mu_4^{-\frac14}\right)\|\nabla
 (\Delta d-f)\|^2+C(1+\mu_4^{-\frac12})\mathcal{A}.\non
 \eea
Putting all the above estimates together, we arrive at the higher-order differential inequality
\eqref{higher
order energy inequality in large viscosity case}. The proof is
complete. \qed
\end{proof}
 \bl
\label{Lemma of the proof of uniform bound in large viscosity} Under
the assumption \textbf{Case I}, for any initial data $(v_0, d_0) \in
V \times \mathbf{H}^2$, if the viscosity $\mu_4$ is properly large,
i.e.,
 $$\mu_4 \geq \mu_4^0(\mu_i, \lambda_1, \lambda_2, v_0, d_0,\underline{\mu}), \quad i=1,2,3,5,6,$$
  we have
 \be
 \mathcal{A}(t) \leq C, \ \ \forall\, t \geq 0. \label{unbdA}
 \ee
  The uniform bound
$C$ is a constant depending only on $f$, $Q$, $\|v_0\|_V$,
$\|d_0\|_{\mathbf{H}^2}$, $\mu's$, $\lambda's$, $\underline{\mu}$.
  \el
\begin{proof}
 It follows from  \eqref{higher order
energy inequality in large viscosity
 case} that
 \bea
 && \frac{d}{dt}\tilde{\mathcal{A}}(t)+\left(\frac{\mu_4}{2}-C_1{\mu_4}^{\frac{1}{2}}\tilde{\mathcal{A}}(t)
\right)\|\Delta v\|^2\non\\
&&\ \
 +\left(-\frac{1}{2\lambda_1}-C_2\mu_4^{-\frac{1}{4}}\tilde{\mathcal{A}}(t)
 \right)
 \|\nabla(\Delta d-f)\|^2\non\\
 &\leq&  C_3\tilde{\mathcal{A}}(t).
 \label{higher order energy inequality in large viscosity case At}
 \eea
 Meanwhile, by \eqref{low2}, we have \be \int_{t}^{t+1}\tilde{\mathcal{A}}(\tau)d\tau \leq
\int_{t}^{t+1}\mathcal{A}(\tau)d\tau+1 \leq M, \ \ \forall\, t \geq 0,
\label{integral with length 1}
 \ee
where $M$ is a positive constant depending only on $\mu_i's$ (except
$\mu_4$), $\lambda_i's$, $\|v_0\|$, $\|d_0\|_{\mathbf{H}^1}$. Now we
choose $\mu_4$ large enough such that
 \be {\mu_4}^{\frac12}
\geq 2C_1(\tilde{\mathcal{A}}(0)+4M+C_3M)+
4\lambda_1^2C_2^2(\tilde{\mathcal{A}}(0)+4M+C_3M)^2+1. \label{large
mu 4}
 \ee
 Applying a similar
argument in \cite[Theorem 4.3]{LL01} (cf. also \cite{LL95, LWX10}),
we can see that $\tilde{\mathcal{A}}(t)$ is uniformly bounded for
all $t \geq 0$ and satisfies
 \be
 \frac{\mu_4}{2}-C_1{\mu_4}^{\frac{1}{2}}\tilde{\mathcal{A}}(t)
\geq 0,\quad
\frac{1}{-2\lambda_1}-C_2\mu_4^{-\frac{1}{4}}\tilde{\mathcal{A}}(t)
 \geq 0, \quad \forall \, t\geq 0.\label{laAb}
 \ee
The proof is complete. \qed
\end{proof}
 Next, we briefly discuss \textbf{Case II}.

  \bc
 \label{higher order energy law in large viscosity case AA}
  For $n=3$,  under the assumption \textbf{Case II},
 the inequality \eqref{higher order energy inequality in large viscosity
 case} still holds.
 \ec
 \begin{proof}
 If Parodi's relation \eqref{lam2} does not hold, i.e.,
  $ \lambda_2+(\mu_2+\mu_3) \neq 0$,  then in the derivation of $\frac{d}{dt}\mathcal{A}(t)$ (see Section \ref{cdA}),
   the first term on the right-hand side of
\eqref{expansion of laplace A d} does not cancel with the first term
on the right-hand side of \eqref{expansion of mu 2 and 3}.
Consequently, there is one extra term:
 \be
(\lambda_2+\mu_2+\mu_3)\int_Q d_jN_i\Delta A_{ij}dx.\non
 \ee
 Besides, since we no longer have \eqref{critical point of lambda 2}
 in \textbf{Case II}, we have to re-investigate the left-hand side of \eqref{NAD}.
Using the $d$ equation \eqref{e3} and integration by parts, we get
 \bea
&&(\lambda_2+\mu_2+\mu_3)\int_Q d_jN_i\Delta A_{ij}dx
\non\\
&=&\frac{\lambda_2+\mu_2+\mu_3}{\lambda_1}\int_Q d_j\nabla_l(\Delta
d_i-f_i)\nabla_l A_{ij}dx \non\\
&&\ \ +\frac{\lambda_2+\mu_2+\mu_3}{\lambda_1}
\int_Q \nabla_l d_j(\Delta d_i-f_i)\nabla_l A_{ij} dx \non\\
&&\ \ +\frac{\lambda_2(\lambda_2+\mu_2+\mu_3)}{\lambda_1}\int_Q
|d_j\nabla_l
A_{ij}|^2dx\non\\
&&\ \ +\frac{\lambda_2(\lambda_2+\mu_2+\mu_3)}{\lambda_1}\int_Q
\nabla_l d_j A_{ik}d_k \nabla_l A_{ij}dx\non\\
&&\ \ + \frac{\lambda_2(\lambda_2+\mu_2+\mu_3)}{\lambda_1}\int_Q d_j
A_{ik}\nabla_l d_k \nabla_l A_{ij}dx. \label{expansion of extra
term}
 \eea
 We now estimate the right-hand side of \eqref{expansion of extra term}. For the first term, we
have
\bea &&\frac{\lambda_2+\mu_2+\mu_3}{\lambda_1}\int_Q
d_j\nabla_l(\Delta
d_i-f_i)\nabla_l A_{ij}dx  \non\\
&\leq& C\|d\|_{\mathbf{L}^{\infty}}\|\nabla(\Delta d-f)\|\|\Delta
v\|\non\\
&\leq &
 C(\|\Delta d-f\|^{\frac{1}{2}}+1)\|\nabla(\Delta d-f)\|\|\Delta v\| \non\\
&\leq& {\mu_4}^{\frac14}(1+\|\Delta d-f\|)\|\Delta
v\|^2+\frac{C}{{\mu_4}^{\frac14}}\|\nabla(\Delta d-f)\|^2 \non\\
&\leq& {\mu_4}^{\frac12}\tilde{\mathcal{A}}\|\Delta
v\|^2+C\mu_4^{-\frac14}\|\nabla(\Delta d-f)\|^2.\label{first extra}
 \eea
 The second term
 can be estimated as  \eqref{EsOM2}, while the fourth and fifth terms are similar to \eqref{mmm56}.
Finally, concerning the third term and the two terms on the
left-hand side of \eqref{NAD}, we infer from \eqref{lama1} and
\eqref{noPa1} that
 \bea
 &&\frac{\lambda_2(\lambda_2+\mu_2+\mu_3)}{\lambda_1}-\frac{(\lambda_2)^2}{\lambda_1}-(\mu_5+\mu_6)\non\\
 & = &
 -\frac{1}{\lambda_1}\big[\lambda_1(\mu_5+\mu_6)-\lambda_2(\mu_2+\mu_3)\big]\non\\
 &<& -\frac{1}{\lambda_1}\Big[-\frac12(\lambda_2-\mu_2-\mu_3)^2-\lambda_2(\mu_2+\mu_3)\Big]\non\\
 &=&
 \frac{1}{2\lambda_1}[(\lambda_2)^2+(\mu_2+\mu_3)^2]\leq 0,\non
 \eea
 which yields
 \be
\Big[\frac{\lambda_2(\lambda_2+\mu_2+\mu_3)}{\lambda_1}-\frac{(\lambda_2)^2}{\lambda_1}-(\mu_5+\mu_6)\Big]
\int_Q |d_j\nabla_l A_{ij}|^2dx\leq 0.\non
 \ee
 Combining the other estimates in the proof of Lemma \ref{higher order energy law in large viscosity
 case}, we obtain the inequality \eqref{higher order energy inequality in large viscosity
 case} under assumption \textbf{Case II}.\qed
 \end{proof}
 \bc
\label{Lemma of the proof of uniform bound in large viscosity AA}
Under the assumption \textbf{Case II}, for any initial data $(v_0,
d_0) \in V \times \mathbf{H}_p^2$, if the viscosity $\mu_4$ is
properly large, i.e.,
 $$\mu_4 \geq \mu_4^0(\mu_i, \lambda_1, \lambda_2, v_0, d_0,\underline{\mu}),\quad i=1,2,3,5,6,$$
   we have
$\mathcal{A}(t) \leq C$ for $t\geq 0$ with $C$ being a constant
depending only on $f$, $Q$, $\|v_0\|_V$, $\|d_0\|_{\mathbf{H}^2}$, $\mu's$,
$\lambda's$ and $\underline{\mu}$.
  \ec

 \subsection{Global existence and uniqueness}

In both \textbf{Case I} and \textbf{Case II}, the uniform estimates
we have obtained in Section 4.2 are independent of the approximation
parameter $m$ and time $t$. This indicates that for both cases,
$(v_m, d_m)$ is a global solution to the approximate problem
\eqref{equation 1 in approximate system}--\eqref{I.C. in approximate
problem}:
 \bea &&v_m \in
L^{\infty}(0, +\infty; V) \cap L^2_{loc}(0, +\infty; \mathbf{H}^2_p),
\non\\
&& d_m \in L^{\infty}(0, +\infty; \mathbf{H}^2_p) \cap L^2_{loc}(0,
+\infty; \mathbf{H}^3_p),\non
 \eea
 which further implies that
 \be
 \partial_tv_m\in L^2_{loc}(0,+\infty; \mathbf{L}^2_p), \quad
 \partial_t d_m\in L^2_{loc}(0,+\infty; \mathbf{H}^1_p).\non
 \ee
The uniform estimates enable us to pass to the limit for $(v_m,
d_m)$ as $m\to \infty$. By a similar argument to \cite{LL95,SL08},
we can show that there exist a pair of limit functions $(v,d)$ satisfying
 \bea && v \in L^{\infty}(0, \infty;
V)\cap L^2_{loc}(0, +\infty; \mathbf{H}^2_p), \label{so1}\\
&& d \in
L^{\infty}(0, +\infty; \mathbf{H}^2_p)\cap L^2_{loc}(0, +\infty;
\mathbf{H}^3_p),\label{so}
 \eea
 such that $(v, d)$ is a weak solution of the system \eqref{e1}--\eqref{I.C. in nonzero case}.
 A bootstrap argument based on Serrin's result \cite{Se} and Sobolev
embedding theorems leads to the existence of classical solutions.
  The uniqueness of solutions to the problem \eqref{e1}--\eqref{I.C. in nonzero case} with
  regularity \eqref{so1}--\eqref{so} can be proved as in \cite[Lemma 2.2]{LWX10}.

In summary, we have
 \bt [Global well-posedness]
 \label{global existence of classical solution} Let $n=3$.  We assume
 that either the conditions in \textbf{Case I} or in \textbf{Case II} are satisfied.
 For any  $(v_0,
d_0)\in V\times \mathbf{H}^2_p$, under the large viscosity
assumption $$\mu_4\geq
 \mu_4^0(\mu_i, \lambda_1, \lambda_2, v_0,d_0,\underline{\mu}),\quad  i=1,2,3,5,6,$$
  the problem \eqref{e1}--\eqref{I.C. in nonzero case}
 admits a unique global solution that satisfies \eqref{so1}--\eqref{so}.
 \et
   Besides,
 we have the following continuous dependence on the initial data:
 \bl
 Suppose that the assumptions in Theorem
 \ref{global existence of classical solution} are satisfied. $(v_i,d_i)$ $(i=1,2)$ are global solutions to the problem
 \eqref{e1}--\eqref{I.C. in nonzero case} corresponding to initial data $(v_{0i}, d_{0i})\in V\times \mathbf{H}^2_p$
 $(i=1,2)$. Then for any $t\in [0,T]$, we have
\begin{eqnarray}
&& \| (v_1-v_2)(t)\|^2+\|(d_1-d_2)(t)\|_{\mathbf{H}^1}^2 \non\\
 &&
 \ \ + \int_0^t\left(\frac{\mu_4}{2}\|\nabla
 (v_1-v_2)(\tau)\|^2+\|\Delta( d_1-d_2)(\tau)\|^2\right) d\tau \non\\
 &\leq&
 2e^{Ct}(\|v_{01}-v_{02}\|^2+\|d_{01}-d_{02}\|_{\mathbf{H}^1}^2),\non
\end{eqnarray}
where $C$ is a constant depending on $\|v_{0i}\|_V,
\|d_{0i}\|_{\mathbf{H}^2}, \mu's, \lambda's$ but not on $t$.
 \el
\br \label{2Da} If in addition, we assume either
$$\text{(i)}\ \ \mu_1=0,\  \lambda_2\neq 0,\quad \ \text{or}\quad \  \text{(ii)}\ \ \mu_1\geq 0, \ \lambda_2=0,$$
 the same result
holds true in $2D$ without the largeness assumption on $\mu_4$. In
case (i), we note that the nonlinearity of the highest-order
vanishes. In particular, this applies to the system \eqref{1 for
Huan's model}--\eqref{3 for Huan's model}, which is a simplified
version of the general Ericksen--Leslie model (cf. \cite{SL08} for
the liquid crystal system with rod-like molecules and
\cite{CR,GW3,LWX10} with general ellipsoid shape). On the other
hand, in case (ii), one can apply the maximum principle for $d$ to
obtain its $\mathbf{L}^\infty$-bound, which makes the proof much
easier (cf. \cite{LL01}). \qed
 \er
 \subsection{Long-time behavior: convergence to equilibrium}
Now we briefly discuss the long-time behavior of the global solution
$(v,d)$ obtained in Theorem \ref{global existence of classical
solution}. First, we have the following decay property:
 \bl \label{vcon} For the global
solutions obtained in Theorem \ref{global existence of classical
solution}, we have
 \be
\lim_{t\rightarrow +\infty} (\|v(t)\|_{V}+ \|-\Delta
 d(t)+f(d(t))\|)=0. \label{vcon1}
 \ee
 \el
 \begin{proof}
 We only consider \textbf{Case I} and the proof for \textbf{Case II}
 is similar. From the basic energy law \eqref{basic energy law at the critical point of lambda
 2}, we see that $\mathcal{A}(t)\in L^1(0,+\infty)$. On the
 other hand,  \eqref{higher order energy inequality in large viscosity
 case} together with \eqref{unbdA} and \eqref{laAb} implies that
 $\frac{d}{dt}\mathcal{A}(t)\leq C$. As a consequence,
 $$\lim_{t\to+\infty}\mathcal{A}(t)=0.$$
 The proof is complete.
 \qed
 \end{proof}

  It easily follows from Lemma \ref{vcon} that
 \begin{proposition} \label{lim}
  Suppose that the assumptions in Theorem
 \ref{global existence of classical solution} are satisfied.
 The $\omega$-limit set of $(v_0,d_0)\in V\times \mathbf{H}^2_p$
 denoted by
 $\omega(v_0,d_0)$ is a non-empty bounded connected subset in $V\times
 \mathbf{H}^2_p$, which is also compact in $\mathbf{L}^2 \times \mathbf{H}^1_p$. Besides, we have
 $$\omega(v_0,d_0)\in \mathcal{S}:=\{(0,d): -\Delta d+ f(d)=0,\
 \text{in} \ Q, \ d(x+e_i)=d(x)\ \text{on}\ \partial Q \}.$$
 \end{proposition}

  Therefore, all asymptotic limit points of the system
 \eqref{e1}--\eqref{e3} satisfy the following reduced stationary problem
 \bea
 v_\infty&=&0,\\
 \nabla P_\infty+\nabla\left(\frac{|\nabla d_\infty|^2}{2}\right)&=&-\nabla d_\infty\cdot \Delta d_\infty,\label{s1a}\\
 -\Delta d_\infty+f(d_\infty)&=&0,\label{s2a}\\
 d_\infty(x)&=&d_\infty(x+e_i),\quad x\in \partial Q,\label{sbd}
 \eea
where in \eqref{s1a}, we used the well-known fact that (cf. \cite{LL95})
$$\nabla\cdot(\nabla
d_\infty\odot\nabla d_\infty)=\nabla\left(\frac{|\nabla
d_\infty|^2}{2}\right) + \nabla d_\infty\cdot \Delta d_\infty.$$
\eqref{s1a} is a constraint equation for $d_\infty$. Since $\mathcal{F}'(d)=f(d)$, if $d_\infty$ is a solution to \eqref{s2a},
then \eqref{s1a} is automatically satisfied because all gradients
can be absorbed into the pressure.

We have already proved that the velocity field $v$ decays to zero in $V$
as $t\nearrow +\infty$ (cf. Lemma \ref{vcon}). On the other
hand, we can only conclude \emph{sequential convergence} for $d$ from
compactness of the trajectory: for any unbounded sequence $\{t_j\}$,
there exist a subsequence $\{t_n\}\nearrow + \infty $ such that
 \be \lim_{t_n\rightarrow +\infty} \|d(t_n)-d_\infty\|_{\mathbf{H}^1}
   =0, \label{secon}
   \ee
 where $d_\infty$ satisfies \eqref{s2a}--\eqref{sbd}.
The convergence of $d$ for the whole time sequence is non-trivial
because in general we cannot expect the uniqueness of critical
points of $E(d)$. In the present case, under the periodic boundary
conditions, one may see that the dimension of the set of stationary
solutions is at least $n$. This is because a shift in each variable
may give another steady state. The convergence of $d$ to a single
equilibrium can be achieved by using the well-known \L
ojasiewicz--Simon approach (cf. L. Simon \cite{S83}). We refer to
\cite{Hu} and the references therein for various generalizations and
applications. To this end, we introduce a suitable \L
 ojasiewicz--Simon type inequality in the periodic setting (cf. e.g., \cite{Hu}).
 \bl[\L ojasiewicz--Simon inequality] \label{ls}
  Let $\psi$ be a critical point of the functional
 $$
 E(d)=\frac12\|\nabla d\|^2 + \int_Q \mathcal{F}(d)dx.$$ Then there exist constants
 $\theta\in(0,\frac12)$ and $\beta>0$ depending on $\psi$ such that
 for any $d\in \mathbf{H}^1_p$ satisfying $\|d-\psi\|_{\mathbf{H}^1}<\beta$, it
 holds
 \be
 \|-\Delta d+f(d)\|_{(\mathbf{H}^1_p)'}\geq
 |E(d)-E(\psi)|^{1-\theta},\label{lsa}
 \ee
 where $(\mathbf{H}^1_p)'$ is the dual space of $\mathbf{H}^1_p$.
 \el
 Then we have the following convergence result:
\bt [Convergence to equilibrium] \label{vconr} Under the assumptions of Theorem \ref{global
existence of classical solution}, the global solution $(v, d)$
has the following property:
 \be
 \|v(t)\|_{V}+\|d(t)-d_\infty\|_{\mathbf{H}^2}\leq C(1+t)^{-\frac{\theta}{(1-2\theta)}}, \quad \forall\ t \geq
 0,\label{rate}
 \ee
 where $d_\infty$ is a solution to \eqref{s2a}--\eqref{sbd}, $C$ is a constant depending on
 $v_0$, $d_0$, $f$, $Q$, $\mu_i's$, $\lambda_i's$, $d_\infty$ and the constant $\theta \in
 (0,\frac12)$ depends on $d_\infty$ (called \L ojasiewicz exponent, cf. Lemma \ref{ls}).
 \et
 Based on Lemma \ref{ls},
 the basic energy law deduced in Section 2 (cf. \eqref{basic energy law at the critical point of lambda 2} or \eqref{belc}) and the higher-order energy
 inequality (cf. Lemma \ref{higher order energy law in large viscosity case}
 or
 Corollary \ref{higher order energy law in large viscosity case AA}), we can prove Theorem \ref{vconr} following  the procedure in \cite[Section 3.2, 3.3]{LWX10} with minor modifications. In order not to make the paper too lengthy, we leave the  details to interested readers.

\section{Well-posedness and Nonlinear Stability under Parodi's Relation }
\setcounter{equation}{0}

The results obtained in Section \ref{Larmu4} indicate that for both
\textbf{Case I} (with Parodi's relation) and \textbf{Case II}
(without Parodi's relation), global well-posedness of the Ericksen--Leslie system can be obtained provided that
the viscosity $\mu_4$ is properly large. Recall the Navier--Stokes equations in $3D$ (with periodic boundary
conditions and $v_0\in H$), we can easily derive
 \be
 \frac{d}{dt}\|\nabla
 v\|^2+\Big(\frac12 \mu_4-{\mu_4}^\frac12\|\nabla v\|^2\Big)\|\Delta
 v\|^2\leq C\mu_4^{-\frac{11}{2}}\|\nabla v\|^2,\non
 \ee
 which implies that the large viscosity assumption is equivalent to small initial
 data assumption on $v$
in $\mathbf{H}^1$-norm. However, this is not the case for the Ericksen--Leslie system
\eqref{e1}--\eqref{I.C. in nonzero case} due to its much more complicated structure (cf. \eqref{large mu 4}). Actually, we do not have the
large viscosity/small initial data alternative relation even for
those simplified liquid crystal systems \cite{LL95,SL08}.

 In this section, we show that Parodi's relation
\eqref{lam2} plays an important role in the well-posedness and
stability of the system \eqref{e1}--\eqref{I.C. in nonzero case}, if
no additional requirement is imposed on the viscosity $\mu_4$. In
particular, under those assumptions in \textbf{Case I}, we are able
to prove a suitable higher-order energy inequality that yields the
local well-posedness and furthermore, the global existence result
provided that the initial velocity $v_0$ is near zero and the
initial director $d_0$ is close to a \textit{local} minimizer
$d^\ast$ of the elastic energy
 \be E(d)=\frac12\|\nabla
d\|^2+\int_{Q}\mathcal{F}(d)dx. \label{ElaE}
 \ee
 Besides, we are able to
show the Lyapunov stability of \emph{local} energy minimizers of
$E(d)$. This implies that Parodi's relation \eqref{lam2}
serves as a sufficient condition for nonlinear stability of the Ericksen--Leslie system \eqref{e1}--\eqref{I.C. in nonzero case} from the mathematical point of view.

\subsection{Higher-order energy inequality and local well-posedness}
 \bl
 \label{HOEL}
 Let $n=3$. Suppose that the conditions in \textbf{Case I} are satisfied.
  Then the following higher-order energy inequality holds:
 \bea
&&
\frac{d}{dt}\mathcal{A}(t)+\frac{\mu_1}{2}\int_{Q}(d_kd_p\nabla_lA_{kp})^2dx+\frac{\mu_4}{8}\|\Delta
v\|^2-\frac{1}{8\lambda_1}\|\nabla(\Delta
d-f)\|^2\non\\
&\leq& C_*(\mathcal{A}^6(t)+\mathcal{A}(t)),
 \label{Eins}
\eea
 where $C_{\ast}$ is a constant that only depends on $\mu's$,
$\lambda's$, $\|v_0\|$ and $\|d_0\|_{\mathbf{H}^1}$.
 \el
\begin{proof}
 First, from the basic energy law \eqref{basic energy law at
the critical point of lambda 2} we still have the uniform estimates
on $\|v(t)\|$ and $\|d(t)\|_{\mathbf{H}^1}$ (cf. \eqref{low1}).
Moreover, estimates \eqref{Agmon}--\eqref{how to control nabla delta
d 2} are still valid. Next, we re-estimate the terms
$I_1,...,I_{14}$ on the right-hand side of \eqref{second time
expansion of derivative of A(t)}.
 \bea I_1
 &\leq& \frac{\mu_1}{4}\int_{Q}(d_kd_p\nabla_lA_{kp})^2dx+
C\|d\|_{\mathbf{L}^\infty}^2\|\nabla v\|^2_{\mathbf{L}^3}\|\nabla d\|_{\mathbf{L}^6}^2. \non\\
&\leq& \frac{\mu_1}{4}\int_{Q}(d_kd_p\nabla_lA_{kp})^2dx+
C(\|\Delta d-f\|^3+1)\|\nabla v\|\|\Delta v\| \non\\
&\leq& \frac{\mu_1}{4}\int_{Q}(d_kd_p\nabla_lA_{kp})^2dx+\frac{\mu_4}{32}\|\Delta v\|^2\non\\
&&\ \
+C(\|\nabla v\|^2\|\Delta d-f\|^{6}+\|\nabla v\|^{2})         \non\\
&\leq& \frac{\mu_1}{4}\int_{Q}(d_kd_p\nabla_lA_{kp})^2dx+\frac{\mu_4}{32}\|\Delta v\|^2+C\mathcal{A}^4+C\mathcal{A},
 \label{first term concerning mu 1}
 \eea
  \bea I_2
&\leq& \frac{\mu_1}{4}\int_{Q}(d_kd_p\nabla_lA_{kp})^2dx+
C\|d\|_{\mathbf{L}^\infty}^2\|\nabla v\|^2_{\mathbf{L}^3}\|\nabla
d\|_{\mathbf{L}^6}^2\non\\
&& \ \
+C\|d\|_{\mathbf{L}^\infty}^3\|\nabla v\|^2_{\mathbf{L}^4}\|\Delta d\|   \non\\
&\leq& \frac{\mu_1}{4}\int_{Q}(d_kd_p\nabla_lA_{kp})^2dx+\frac{\mu_4}{32}\|\Delta
v\|^2+C\mathcal{A}^4+C\mathcal{A}\non\\
&& \ \ +C(\|\Delta
d-f\|^\frac52+1)\|\nabla v\|^\frac12\|\Delta v\|^\frac32   \non\\
&\leq&
\frac{\mu_1}{4}\int_{Q}(d_kd_p\nabla_lA_{kp})^2dx+\frac{\mu_4}{16}\|\Delta
v\|^2+C\mathcal{A}^6+C\mathcal{A},\non
 \eea
 \bea
I_3+I_4 &\leq& C\|\Delta v\|\|\nabla v\|_{\mathbf{L}^3}\|\nabla d\|_{\mathbf{L}^6}\|d\|_{\mathbf{L}^\infty}\non\\
& \leq&  \frac{\mu_4}{64}\|\Delta v\|^2+ C\|\nabla v\|_{\mathbf{L}^3}^2\|\nabla d\|_{\mathbf{L}^6}^2\|d\|_{\mathbf{L}^\infty}^2  \non\\
&\leq& \frac{\mu_4}{32}\|\Delta v\|^2+C\mathcal{A}^4+C\mathcal{A},
\label{estimate of mu 5 plus mu 6 in small initial data}
 \eea
 \bea
I_5&\leq&
C\|\nabla(\Delta d-f)\|\|\nabla v\|^{\frac12}\|\Delta
v\|^{\frac12}(\|\Delta d-f\|+1 )
\non\\
&\leq& \frac{\mu_4}{32}\|\Delta
v\|^2-\frac{1}{8\lambda_1}\|\nabla(\Delta d-f)\|^2+C\|\nabla
v\|^2(\|\Delta d-f\|^4+1)    \non\\
&\leq& \frac{\mu_4}{32}\|\Delta
v\|^2-\frac{1}{8\lambda_1}\|\nabla(\Delta
d-f)\|^2+C\mathcal{A}^3+C\mathcal{A}, \label{estimate of omega
multiplying nabla d in small initial data}
\eea
\bea
I_6&\leq&C\|\Delta v\|\|\Delta d-f\|(\|(\nabla (\Delta d-f)\|^\frac34+1)  \non\\
 &\leq& \frac{\mu_4}{32}\|\Delta
v\|^2-\frac{1}{8\lambda_1}\|\nabla(\Delta
d-f)\|^2+C\mathcal{A}^4+C\mathcal{A}. \label{estimate of nabla omega
multiplying nabla d in small data}
 \eea
For the terms $K_1,...,K_4$ in \eqref{NAd1}, we still have
\eqref{NAD}. By a similar argument to \eqref{estimate of omega
multiplying nabla d in small initial data}--\eqref{estimate of nabla
omega multiplying nabla d in small data} we get
 \bea
K_1 &\leq& \frac{\mu_4}{32}\|\Delta
v\|^2-\frac{1}{8\lambda_1}\|\nabla(\Delta
d-f)\|^2+C\mathcal{A}^4+C\mathcal{A},\non\\
 K_2 &\leq& \frac{\mu_4}{32}\|\Delta v\|^2
-\frac{1}{8\lambda_1}\|\nabla(\Delta
d-f)\|^2+C\mathcal{A}^3+C\mathcal{A}.\non
  \eea
Using integration by parts, we can see that
   \bea K_3&=&-\frac{2(\lambda_2)^2}{\lambda_1}\int_Q \nabla_lA_{ij}
   d_j A_{ij} \nabla_l d_j dx \non\\
   &=& \frac{2(\lambda_2)^2}{\lambda_1}\int_Q |A_{ij}\nabla_l
   d_j|^2dx-K_3+\frac{2(\lambda_2)^2}{\lambda_1}\int_Q A_{ij}d_j
   A_{ij}\Delta d_j dx,\non
 \eea
which together with similar estimates in \eqref{estimate of mu 5
plus mu 6 in small initial data} yields that
 \bea
 K_3&=& \frac{(\lambda_2)^2}{\lambda_1}\int_Q |A_{ij}\nabla_l
   d_j|^2dx+\frac{(\lambda_2)^2}{\lambda_1}\int_Q A_{ij}d_j
   A_{ij}\Delta d_j dx\non\\
   &\leq& \frac{\mu_4}{32}\|\Delta
   v\|^2+C\mathcal{A}^4+C\mathcal{A}.\non
 \eea
Hence,
 \bea && I_7+I_8+I_9 \non\\
 &\leq& \frac{3\mu_4}{32}\|\Delta v\|^2
-\frac{1}{4\lambda_1}\|\nabla(\Delta
d-f)\|^2+(\mu_5+\mu_6)\int_{Q}|\nabla_l A_{ij} d_j|^2dx\non\\
&& +C\mathcal{A}^4+C\mathcal{A}.  \non
 \eea
The remaining terms can be estimated in a straightforward way.
 \bea
 I_{10} &\leq& |(\Delta v, v\cdot \nabla v)| \leq C\|\Delta
v\|^\frac74\|\nabla v\|\non\\
&\leq & \frac{\mu_4}{32}\|\Delta v\|^2+C\|\nabla
v\|^8, \non
\eea
\bea
I_{11}& \leq& C(\|d\|_{\mathbf{L}^6}^2+1)\|\Delta
d-f\|_{\mathbf{L}^3}^2 \non\\
& \leq & -\frac{1}{8\lambda_1}\|\nabla(\Delta
d-f)\|^2+C\mathcal{A},\non
\eea
\bea
I_{12} & \leq&
C\|f'(d)d\|\|\Delta d-f\|_{\mathbf{L}^6}\|\nabla v\|_{\mathbf{L}^3}  \non\\
&\leq& C\|\nabla v\|^{\frac12}\|\Delta
v\|^{\frac12}\Big(\|\nabla(\Delta d-f)\|+\|\Delta d-f\| \Big)
\non\\
&\leq& \frac{\mu_4}{32}\|\Delta
v\|^2-\frac{1}{8\lambda_1}\|\nabla(\Delta d-f)\|^2+C\mathcal{A}.\non
 \eea
The estimate of $I_{13}$ is similar to \eqref{estimate of omega
multiplying nabla d in small initial data}:
 \be
I_{13} \leq -\frac{1}{8\lambda_1}\|\nabla(\Delta
d-f)\|^2+C\mathcal{A}^3+C\mathcal{A}.\non
 \ee
 For the last term $I_{14}$, we have
 \bea
 I_{14}&\leq& C\|\Delta d-f\|_{\mathbf{L}^3}\|v\|_{\mathbf{L}^6}\|\nabla f\|
 \non\\
 &\leq& C(1+\|\Delta d-f\|)(\|\nabla (\Delta d-f)\|+\|\Delta d-f\|)\|\nabla
 v\|\non\\
 &\leq& -\frac{1}{8\lambda_1}\|\nabla(\Delta
d-f)\|^2+C\mathcal{A}^2+C\mathcal{A}.\non
 \eea
Collecting all the estimates above, we can conclude the higher-order differential inequality \eqref{Eins}. The proof is complete. \qed
\end{proof}
The following local well-posedness result is a direct consequence of
the higher-order energy inequality \eqref{Eins}:
\bt[Local well-posedness]
 \label{locals solution} Let $n=3$. Suppose that the conditions in \textbf{Case I} are satisfied.
 For any  $(v_0,
d_0)\in V\times \mathbf{H}^2_p$, there exists a $T^*>0$ such that
the problem \eqref{e1}--\eqref{I.C. in nonzero case} admits a unique
local solution satisfying
 $$ v \in L^{\infty}(0, T^*; V)\cap L^2(0,
T^*; \mathbf{H}^2_p),\quad d \in L^{\infty}(0, T^*; \mathbf{H}^2_p)\cap
L^2(0, T^*; \mathbf{H}^3_p).$$
 \et

\br
 Unfortunately, we are not able to prove a corresponding local well-posedness result under the assumptions in
 \textbf{Case II} where Parodi's relation \eqref{lam2} is \textit{not} satisfied. In this case the higher-order energy inequality
 \eqref{Eins} is not available any longer. One obvious difficulty is that we
 lose control of some higher-order nonlinearities that will vanish due to
specific cancellations under Parodi's relation
  (see, e.g., \eqref{first extra}). \qed
\er

\subsection{Near local minimizers: well-posedness and nonlinear stability }

Based on Lemma \ref{HOEL}, one can easily deduce the following
property:

\begin{proposition}
 \label{sm}
 Suppose  that the assumptions in \textbf{Case I} are satisfied. For any $(v_0, d_0)\in V\times \mathbf{H}^2_p$,
 if
 \be
 \|\nabla v\|^2(0)+\|\Delta
d-f(d)\|^2(0)\leq R,\label{R}
 \ee
 where $R>0$ is a constant, there exists a positive constant
$\varepsilon_0$ depending on $\mu's$, $\lambda's$, $\|v_0\|$,
$\|d_0\|_{\mathbf{H}^1}$, $f$, $Q$ and $R$, such that the following
property holds: for the (unique) local solution $(v,d)$ of the
system \eqref{e1}--\eqref{I.C. in nonzero case} which exists on
$[0,T^*]$, if
 $$\mathcal{E}(t) \geq \mathcal{E}(0)-\varepsilon_0, \quad
\forall\, t\in[0,T^*],$$
 then the local solution $(v,d)$ can be
extended beyond $T^*$.
\end{proposition}
\begin{proof}
 We consider the following initial
value problem of an ordinary differential equation:
 \be
 \frac{d}{dt}Y(t)=C_*(Y(t)^6+Y(t)),\quad
Y(0)= R\geq \mathcal{A}(0).\label{ODE}
 \ee
 We denote by
$I=[0,T_{max})$ the maximal existence interval of $Y(t)$ such that
$$ \lim_{t\rightarrow T_{max}^-} Y(t)=+\infty.$$
 It follows from the comparison principle that for any $t\in I$, $0\leq \mathcal{A}(t)\leq Y(t)$.
 Consequently, $\mathcal{A}(t)$ exists on $I$. We note that
 $T_{max}$ is determined by $Y(0)=R$ and $C_*$ such that $T_{max}=T_{max}(R,C_*)$ is
 increasing as $R$ decreases. Taking $$t_0=\frac34 T_{max}(R, C_*)> 0,$$
  then
 we have
 \be
 0\leq \mathcal{A}(t)\leq Y(t)\leq K, \quad \forall\, t\in [0,
 t_0],\label{uniK}
 \ee
 where $K$ is a constant that only
depends on $R, C_*, t_0$. This fact combined with the Galerkin
approximate scheme in Section 4.1 leads to the local existence of a
unique solution to the system \eqref{e1}--\eqref{I.C. in nonzero
case} at least on $[0,t_0]$. (This indeed provides a proof of Theorem
\ref{locals solution}.)

The above argument suggests that the existing time $T^*\geq t_0$.
Now if $ \mathcal{E}(t) \geq \mathcal{E}(0)-\varepsilon_0$ for all
$t\in[0,T^*]$, we infer from Lemma \ref{BEL} that
 \[\int_0^{T^*}\int_{Q}\Big(\frac{\mu_4}{2}|\nabla v(t)|^2-\frac{1}{\lambda_1}|\Delta d(t)-f(d(t))|^2\Big)\,dxdt
 \leq \varepsilon_0.
  \]
  Hence, there exists a
$t_* \in [T^*-\frac{t_0}{3}, T^*]$ such that
 \[ \|\nabla v(t_*)\|^2+\|\Delta
d(t_*)-f(d(t_*))\|^2 \leq \max\Big\{\frac{2}{\mu_4},
-\lambda_1\Big\}\frac{3\varepsilon_0}{t_0}.
 \] Choosing $\varepsilon_0>0$  such that
 \be
 \max\Big\{\frac{2}{\mu_4},
-\lambda_1\Big\}\frac{3\varepsilon_0}{t_0}= R,\label{epsilon}
 \ee
 we have
$\mathcal{A}(t_*)\leq R$. Taking $t_*$ as the initial time and
$Y(t_*)=R$ in \eqref{ODE}, we infer from the above argument that
$Y(t)$ (and thus $\mathcal{A}(t)$) is uniformly bounded at least on
$[0,t_*+t_0]\supset [0, T^*+\frac23 t_0]$. Thus, we can extend the
local solution $(v, d)$ from $[0, T^*]$ to $[0, T^*+\frac23 t_0]$. The proof is complete.
\qed
\end{proof}

\br

Proposition \ref{sm} implies that, for the local solution $(v,d)$ of
\eqref{e1}--\eqref{I.C. in nonzero case}, if the total energy
$\mathcal{E}(t)$ does not drop too much on its existence interval
$[0,T^*]$, then it can be extended beyond $T^*$. We note that
stronger results have been obtained in \cite{LL95,LL01} for
simplified liquid crystal systems. In those cases,  global
existence of weak solutions can be proved and the total energy
$\mathcal{E}(t)$ is well-defined on $[0,+\infty)$. Then one can show
the alternative relation: either there exists a $T<+\infty$ such
that $\mathcal{E}(T) <\mathcal{E}(0)-\varepsilon_0$ or the system
admits a (unique) global strong solution.\qed
 \er

 \br It is easy to verify that the above hypothesis on the changing
rate of $\mathcal{E}(t)$ can be fulfilled, if the initial velocity
$v_0$ is near zero and the initial molecule director $d_0$ is close
to an \emph{absolute} minimizer of the elastic energy $E(d)$ (for instance, a constant vector with unit length). We
refer to \cite{LL95,LL01,LW08} for the cases of the simplified
liquid crystal system. The same result holds for our current general
case if the same assumption is imposed.\qed
 \er

The assumption that the initial director  $d_0$ is close to an
\emph{absolute} energy minimizer can indeed be improved. Under Parodi's
relation \eqref{lam2}, we can show much stronger result that if
$v_0$ is near zero and $d_0$ is
 close to a \emph{local} minimizer of $E(d)$, then the
total energy $\mathcal{E}$ will never drop too much. Actually, we
shall see below that the global solution will stay close to the
given minimizer for all time (i.e., Lyaponov stability) and
$\mathcal{E}(t)$ will converge to the same energy level of the
\textit{local} minimizer. This generalized result also applies to
all those simplified Ericksen--Leslie systems considered in the literature
\cite{LL95,LL01,SL08,LWX10,LW08}.

 \bd \label{definition of local minimizer} $d^\ast \in \mathbf{H}^1_p$ is
called a local minimizer of $E(d)$, if there exists $\sigma > 0$,
such that for any $d \in \mathbf{H}^1_p$ satisfying
$\|d-d^\ast\|_{\mathbf{H}^1} \leq \sigma$, it holds $E(d) \geq
E(d^\ast)$.
 \ed
 \br
  Since any minimizer of $E(d)$ is also a
 critical point of $E(d)$, it satisfies the Euler--Lagrange equation
 \be  - \Delta d + f(d)=0,\quad x\in Q, \quad   d(x)=d(x+e_i),\quad x\in \partial Q.
   \label{staaq}
 \ee
 From the elliptic regularity theory and bootstrap argument, one can easily see that
 if the solution $d\in
\mathbf{H}^1_p$, then $d$ is smooth.\qed
 \er
Next, we state the main result of this section:
\begin{theorem} \label{Main theorem III}
Suppose that $n=3$ and the conditions in \textbf{Case I} are
satisfied. Let $d^\ast \in \mathbf{H}^2_p$ be a local minimizer
of $E(d)$. There exist positive constants $\sigma_1, \sigma_2$,
which may depend on $\lambda_i's$, $\mu_i's$, $Q$, $\sigma$ and
$d^\ast$, such that for any initial data $(v_0, d_0) \in V \times
\mathbf{H}^2_p$ satisfying $$\|v_0\|_{\mathbf{H}^1} \leq 1,\quad
\|d_0-d^\ast\|_{\mathbf{H}^2} \leq 1$$ and $$\|v_0\| \leq \sigma_1,\quad
\|d_0-d^\ast\|_{\mathbf{H}^1} \leq \sigma_2,$$
 we have

(i) the problem \eqref{e1}--\eqref{I.C. in nonzero case} admits a
unique global solution $(v,d)$,

(ii) $(v,d)$ enjoys the same long-time behavior as in Theorem
\ref{vconr}. In addition,
 \be
  \lim_{t\to+\infty}\mathcal{E}(t)=E(d_\infty)=E(d^\ast).\label{leee}
 \ee
\end{theorem}
\begin{proof}
Without loss of generality, we assume that the constant $\sigma$ in
Definition \ref{definition of local minimizer} satisfies $
\sigma\leq 1$. Throughout the proof, $C_i$, $i=1, 2,\cdots $ denote
generic constants depending only on $\mu_i's$, $\lambda_i's$,
$\sigma$ and $d^\ast$.
By our assumptions, we easily see that
 \bea
 && \|v(t)\|+\|d(t)\|_{\mathbf{H}^1}\leq C_1,\quad \forall\, t\geq 0, \label{locunil}\\
 && \mathcal{A}(0) = \|\nabla v_0\|^2+\|\Delta
d_0-f(d_0)\|^2 \leq  C_2.
 \eea

Recalling the proof of Proposition \ref{sm}, we take $R=C_2$ for our
current case. The constant $C_*$ in \eqref{ODE} can be determined by
$C_1$ and $\mu's, \lambda's$ (cf. Lemma \ref{HOEL}). Then we set
$t_0=\frac34T_{max}(C_2, C_*)$ and take $T^*=t_0$. Finally, the
critical constant $\varepsilon_0$ is given by \eqref{epsilon}. It
follows from \eqref{uniK} that $\mathcal{A}(t)$ is uniformly bounded
on $[0, t_0]$, which implies
 \be
  \|v(t)\|_V+\|d(t)\|_{\mathbf{H}^2}\leq C_3,\quad \forall\, t\in [0,t_0].
  \label{hie}
  \ee


Next, we extend the local solution to $[0,+\infty)$ by using the \L
ojasiewicz--Simon approach. Since the minimizer $d^*$ is a critical
point of $E(d)$, we take $\psi=d^*$ in the \L ojasiewicz--Simon
inequality (cf. Lemma \ref{ls}), then the constants $\beta>0,
\theta\in (0, \frac12)$ are determined by $d^*$ and \eqref{lsa}
holds.

 The proof consists of several steps.

 \textbf{Step 1.}  In order to apply Proposition \ref{sm} with $T^*=t_0$, it suffices to show that
 \be
  \mathcal{E}(t)-\mathcal{E}(0) \geq -\varepsilon_0, \quad \forall\, t\in [0, t_0].
  \label{tedro}
 \ee
 Using \eqref{locunil} and the Sobolev embedding theorems, we have $$|E(d_0)-E(d^\ast)| \leq C_4\|d_0-d^\ast\|_{\mathbf{H}^1},$$
 which implies that
 \bea
\mathcal{E}(t)-\mathcal{E}(0)&=&\frac12\|v(t)\|^2-\frac12\|v_0\|^2+E(d(t))-E(d_0)
\non\\
&\geq&-\frac12\|v_0\|^2+E(d(t))-E(d^\ast)+E(d^\ast)-E(d_0) \non\\
&\geq& -\frac12\|v_0\|^2-C_4\|d_0-d^\ast\|_{H^1}
+E(d(t))-E(d^\ast).\label{tedrop1}
 \eea
 Take
 \be
 \sigma_1\leq \min\Big\{\varepsilon_0^\frac12, 1\Big\}, \quad \sigma_2\leq
 \min\Big\{\frac{\varepsilon_0}{2C_4},1\Big\}.\non
 \ee
 Then by \eqref{tedrop1}, it is easy to check \eqref{tedro} will be satisfied provided
 that
 \be
  E(d(t)) - E(d^\ast)\geq 0, \quad \forall \,  t \in[0, t_0].  \label{nearc}
 \ee
 By the definition of $d^\ast$, it reduces to prove that
\be
 \|d(t)-d^\ast\|_{\mathbf{H}^1}\leq \sigma, \quad \forall \,  t \in[0,
 t_0].\label{nlm}
 \ee
 Actually, we can prove a slightly stronger conclusion such that
 \be
 \|d(t)-d^\ast\|_{\mathbf{H}^1}< \omega:=\frac12\min\{\sigma, \beta\}, \quad \forall \,  t \in[0,
 t_0].\label{nlma}
 \ee
 Suppose $$\sigma_2\leq \frac14  \omega.$$
 We use a contradiction argument. If \eqref{nlma} is not true, then by the continuity of $d$ that
$d\in C([0,t_0]; \mathbf{H}^1)$, there exists a minimal time $T_0\in
(0, t_0]$, such that $$\|d(T_0)-d^\ast\|_{\mathbf{H}^1}=\omega.$$
 Observe that $$\mathcal{E}(t) =\frac12\|v(t)\|^2+E(d(t))\geq
E(d^\ast), \quad \forall \, t \in [0, T_0].$$

  First, we consider the trivial
case that for some $T\leq T_0$, $\mathcal{E}(T)=E(d^\ast)$. Then
we deduce from the definition of the local minimizer that for $t\geq
T$, $\mathcal{E}$ cannot drop and will remain $E(d^\ast)$. Thus, we
infer from the basic energy law \eqref{basic energy law at the
critical point of lambda 2} that the evolution will be stationary
and the conclusion easily follows.

In the following, we just assume
$\mathcal{E}(t) > E(d^\ast)$ for $t \in [0, T_0]$. Applying Lemma
\ref{ls} with $\psi=d^\ast$, we get
 \bea  && -\frac{d}{dt}[\mathcal{E}(t)-E(d^\ast)]^\theta\non\\
 &=&-\theta[\mathcal{E}(t)-E(d^\ast)]^{\theta-1}\frac{d}{dt}\mathcal{E}(t)\non\\
&\geq& \frac{\theta\left(\frac{\mu_4}{2}\|\nabla
v\|^2-\frac{1}{\lambda_1}\|\Delta d-f\|^2
\right)}{C(\|v\|^{2(1-\theta)}+\|\Delta d-f\|)}
\non\\
&\geq& C_5(\|\nabla v\|+\|\Delta d-f\|), \quad \forall \, t\in (0,
T_0).\non
 \eea
On the other hand, it follows from \eqref{e3} and \eqref{hie} that
 \bea \|d_t\| &\leq&
\|v\cdot\nabla d\|+\|\Omega
d\|+\Big|\frac{\lambda_2}{\lambda_1}\Big|\|Ad\|-\frac{1}{\lambda_1}\|\Delta
d-f\| \non\\
&\leq& C_6(\|v\|_{\mathbf{L}^6}\|\nabla d\|_{\mathbf{L}^3}+\|\nabla
v\|\|d\|_{\mathbf{L}^\infty}+\|\Delta d-f\|) \non\\
&\leq& C_7(\|\nabla v\|+ \|\Delta d-f\|), \quad \forall \, t\in
[0,t_0].
 \eea
Consequently,
 \bea
 \|d(T_0)-d_0\|_{\mathbf{H}^1} &\leq&
 C_8\|d(T_0)-d_0\|^{\frac12}\|d(T_0)-d_0\|_{\mathbf{H}^2}^{\frac12}\non\\
 &\leq& C_9\Big( \int_{0}^{T_0}\|d_t(t)\|dt \Big)^\frac12 \leq
C_{10}[\mathcal{E}(0)-E(d^\ast)]^{\frac{\theta}{2}}\non\\
&  \leq & C_{10}\Big(
\frac12\|v_0\|^2+C_4\|d_0-d^\ast\|_{\mathbf{H}^1} \Big)^\frac{\theta}{2} \non\\
&\leq& C_{11}\big(\|v_0\|^{\theta}+
\|d_0-d^\ast\|_{\mathbf{H}^1}^\frac{\theta}{2}\big).\label{dddif}
 \eea
Finally, choosing (also taking the previous assumptions into
account)
 \be \sigma_1=\min\left\{\varepsilon_0^\frac12,
\Big(\dfrac{\omega}{4C_{11}}\Big)^{\frac{1}{\theta}}, 1\right\}, \ \
\sigma_2=\min\left\{\frac{\varepsilon_0}{2C_4},
\Big(\dfrac{\omega}{4C_{11}}\Big)^{\frac{2}{\theta}},
\frac{\omega}{4}, 1\right\}, \label{bound of delta 1 and delta 2}
 \ee
 we can deduce from \eqref{dddif} that
 \bea \|d(T_0)-d^\ast\|_{\mathbf{H}^1}
 &\leq&
\|d(T_0)-d_0\|_{\mathbf{H}^1}+\|d_0-d^\ast\|_{\mathbf{H}^1}\non\\
& \leq&
\frac{\omega}{4}+\frac{\omega}{4}+\frac{\omega}{4} < \omega, \non
\eea
which leads to a contradiction with the definition of $T_0$. Thus,
\eqref{nlma} is true and so is \eqref{nearc}, which implies that
\eqref{tedro} is satisfied.

As in the proof of Proposition \ref{sm}, there exists a $t_{\ast}
\in [\frac{2t_0}{3}, t_0]$, such that $\mathcal{A}(t^*)\leq R$. Then
we conclude that $\mathcal{A}(t)$ is uniformly bounded on $[0,
t^*+t_0]\supset [0, \frac{5t_0}{3}]$ (with \emph{the same bound} as
on $[0,t_0]$). Here, we note the important fact that the bound of
$\mathcal{A}(t)$ only depends on $R, C_*, t_0$ but not on the length
of existence interval.

\textbf{Step 2}. Now we take $T^*=\frac53t_0$. By the same argument
as in Step 1, we can show that $$
  \mathcal{E}(t)-\mathcal{E}(0) \geq -\varepsilon_0, \quad t\in [0,
  T^*].$$
   Again, we obtain that $\mathcal{A}(t)$ is uniformly bounded on $[0, T^*+\frac23t_0]$ (with \emph{the same
bound} as on $[0,t_0]$). By iteration, one can see that the local
solution can be extended by a fixed length $\frac23 t_0$ at each
step and $\mathcal{A}(t)$ is uniformly bounded by a constant only
depending on $R, C_*, t_0$.

Therefore, we can show that $(v, d)$ is indeed a global solution.
Moreover, the following uniform estimate holds
 \be
 \|v(t)\|_{\mathbf{H}^1}+\|d(t)\|_{\mathbf{H}^2}\leq K,\quad \forall\, t\geq
 0, \label{uniKK}
 \ee
 where $K$ depends on $C_1$, $R, C_*, t_0$. The conclusion (i) is proved.

\textbf{Step 3}. Based on the uniform estimate \eqref{uniKK}, a similar argument to
Theorem \ref{vconr} yields that there exists a $d_\infty$ satisfying
\eqref{s2a}--\eqref{sbd}, such that
 \be \lim_{t \rightarrow
+\infty}(\|v(t)\|_{V}+ \|d(t)-d_\infty\|_{\mathbf{H}^2})=0,
\label{convvv}
 \ee
 with the convergence rate
\eqref{rate} (We remark that in \eqref{rate}, the \L ojasiewicz
exponent $\theta$ is determined by the limiting function $d_\infty$,
which is \emph{different} from the one we have used in Step 1).

By repeating the argument in Step 1, we are able to show that
$$\|d(t)-d^\ast\|_{\mathbf{H}^1}\leq \omega, \quad \forall\, t\geq 0.$$
 Then
for $t$ sufficiently large, we have
 \bea
 \|d_\infty-d^\ast\|_{\mathbf{H}^1} &\leq&
 \|d_\infty-d(t)\|_{\mathbf{H}^1}+\|d(t)-d^\ast\|_{\mathbf{H}^1}\non\\
 &\leq&
 \frac32 \omega< \min\{\beta, \sigma\}.\label{indlo}
 \eea
 Applying Lemma \ref{ls} again with $d=d_\infty$ and
 $\psi=d^\ast$, we obtain \be |E(d_\infty)-E(d^\ast)|^{1-\theta}\leq
 \|-\Delta d^\ast+f(d^\ast)\|=0,\label{EsED}
 \ee
 which together with \eqref{convvv} yields \eqref{leee}. The proof
 is complete.\qed
\end{proof}
 \br
 We note that in the assumptions $\|v_0\|_{\mathbf{H}^1}
\leq 1$ and $\|d_0-d^\ast\|_{\mathbf{H}^2} \leq 1$, the bound $1$ is
not essential and it can be replaced by any fixed positive constant
$M$. In this case those constants in the proof of Theorem
\ref{Main theorem III} may also depend on $M$.\qed
 \er
 \bc[Nonlinear stability]
Suppose that $n=3$ and the conditions in \textbf{Case I} are
satisfied. Let $d^\ast \in \mathbf{H}^2_p$ be a local minimizer
of $E(d)$. Then  $d^\ast$ is Lyapunov stable.
\ec
 \begin{proof}
 We observe that in the proof of Theorem \ref{Main
theorem III}, $\omega$  can be an arbitrarily small positive constant
satisfying $\omega\leq \frac12\min\{\sigma, \beta\}$, by our choice
of $\sigma_1, \sigma_2$, we actually have shown that the local
minimizer $d^\ast$ is {\rm Lyapunov stable}.\qed
 \end{proof}
 \br
   We can see from
\eqref{indlo} and \eqref{EsED} that the asymptotic limit $d_\infty$ obtained in Theorem \ref{Main theorem III} (ii) is also a
 local minimizer of $E(d)$ (having the same energy level as $d^\ast$). Moreover, if $d^\ast$ is an \textit{isolated} local minimizer, then $d_\infty=d^\ast$
and $d^\ast$ is  \textit{asymptotically stable}.\qed
 \er

\section{Parodi's Relation and Linear Stability}
\setcounter{equation}{0}
In Section 5.2 we have shown that Parodi's relation can be viewed as a sufficient condition for the nonlinear stability of the Ericksen--Leslie system \eqref{e1}--\eqref{e3}. It is still an open problem whether similar result holds true for the original Ericksen--Leslie system \eqref{the most primitive equ 1}--\eqref{the most
primitive equ 3}. Alternatively, in this section we shall make a preliminary study to discuss the connection between Parodi's
relation and linear stability of the original Ericksen--Leslie system \eqref{the most primitive equ 1}--\eqref{the most
primitive equ 3}.

 For the sake of simplicity, we assume that $$\rho=1,\quad  \rho_1=0,\quad
F=G=0$$ and the Oseen--Frank energy density function takes the
simple form $$W=\frac{1}{2}|\nabla d|^2.$$ Besides, we are interested
in the bulk properties of the Ericksen--Leslie system and consider the problem
in the whole space $\mathbb{R}^3$, neglecting the boundary effects.
 We thus simply set the Lagrangian multiplier $\beta=0$ since $\beta$
does not enter into the local field equations and can be determined
through the boundary conditions, if any, on the director stress (cf.
\cite{C73a}). The nematics usually adopt a constant orientation in
uniform shear flow. The analysis in \cite[Section 6]{Le68} shows that for a
material that aligns in shear flow the viscous coefficients must
satisfy $$|\mu_5-\mu_6|\geq |\mu_2-\mu_3|.$$
Here, we assume that \eqref{lama1} and
\eqref{lama1a} are satisfied, then we have
 \be
 |\mu_5-\mu_6|\geq \mu_3-\mu_2>0.
\label{lambda one is less than lambda two}
 \ee

Consider the basic uniformly-oriented equilibrium state in which the
material is at rest (zero velocity), the orientation is uniformly
parallel to a constant unit vector $n=(n_{1}, n_{2}, n_{3})^T$, the
hydrostatic pressure $\tilde{p}$ is a constant and the director
tension (Lagrangian multiplier) $\gamma$ is zero. The equilibrium
state is disturbed by perturbations with a small amplitude: velocity
field $v$, director $d+n$, pressure $\bar{p}+\tilde{p}$ and director
tension $\bar{\gamma}$. After a direct computation, the linearized
equations of the Ericksen--Leslie system \eqref{the most primitive equ
1}--\eqref{the most primitive equ 3} for $(v, d)$ are (cf. e.g.,
\cite{C73})
 \bea
&&\frac{\partial v_i}{\partial t}+\bar{p}_{,i}-
\mu_1{n}_{i}{n}_{j}{n}_{k}{n}_{l}v_{j,kl}
 -\frac{\mu_2+\mu_5}{2}{n}_{j}{n}_{k}v_{k,ij}-\frac{\mu_3+\mu_6}{2}{n}_{i}{n}_{k}v_{k,jj}
 -\frac{\mu_4}{2}v_{i, jj}
 \non\\
 &&\quad -\frac{\mu_5-\mu_2}{2}{n}_{j}{n}_{k}v_{i,kj}-\mu_2{n}_{j}\frac{\partial d_{i,j}}{\partial t}
  -\mu_3{n}_{i}\frac{\partial d_{j,j}}{\partial t}=0, \label{lev} \\
  &&v_{i,i}=0, \label{imc}\\
&&-\lambda_1\frac{\partial d_i}{\partial
t}-\bar{\gamma}{n}_{i}-d_{i, jj}
+\frac{\lambda_1-\lambda_2}{2}{n}_{j}v_{i,j}-\frac{\lambda_1+\lambda_2}{2}{n}_{j}v_{j,i}=0,
\label{led}\\
&& d_in_{i}=0\label{led2}.
 \eea

We study the behavior of
infinitesimal, sinusoidal disturbances by a linear stability analysis. For this purpose, we seek plane wave solutions to the linearized system \eqref{lev}--\eqref{led2} of the
following form (see e.g., \cite{C74})
 \bea
d&=&\mathbf{a}e^{\sqrt{-1}(m\nu\cdot x-\omega t)},\label{plane wave
solution for director}\\
 v&=&\mathbf{b}e^{\sqrt{-1}(m\nu\cdot x-\omega
t)}, \label{plane wave solution for velocity}
\\
\bar{\gamma}&=&Ce^{\sqrt{-1}(m\nu\cdot x-\omega t)},  \\
\bar{p}&=&De^{\sqrt{-1}(m\nu\cdot x-\omega t)}.
 \eea
where  $m$ is the complex wave number, $\nu=(\nu_1,\nu_2,\nu_3)^T$
is a given unit vector specifying the direction of propagation of
the wave and $\omega$ is the complex frequency number.
$\mathbf{a}$, $\mathbf{b}$ are two constant vectors and $C$ and $D$ are two constants.
Due to  the constraint on
the unit length of the director \eqref{led2} and the incompressibility condition \eqref{imc}, the constant vectors $\mathbf{a}$ and
$\mathbf{b}$ satisfy
 \be n \cdot \mathbf{a}=0, \quad  \nu\cdot\mathbf{b}=0. \label{plane wave equation constraints}
 \ee
Let $0 \leq \theta \leq \frac{\pi}{2}$ be a constant angle such that
$\sin\theta=\nu\cdot n$. We consider the in-plane mode and deduce from \eqref{plane wave equation constraints} that (cf. \cite{C74})
\bea \mathbf{a}&=&A(\nu-n\sin\theta),   \label{expansion
of n}\\
\mathbf{b}&=&B(n-\nu\sin\theta),  \label{expansion of v}
 \eea
 where $A$ and $B$ are two constants.
Inserting \eqref{expansion of n} and \eqref{expansion of v} into the linearized system \eqref{lev}--\eqref{led2},
 after direct but tedious computations, we obtain
\bea &&\left( m^2+\sqrt{-1}\lambda_1\omega
\right)A-\frac{\sqrt{-1}m}{2}q(\theta)B=0,
 \label{third equation of Currie 74}  \\
&&m\omega p(\theta)A+\left(\frac{m^2g(\theta)}{2}-\sqrt{-1}\omega
\right)B=0, \label{equation 2 of currie} \\
&&C\sqrt{-1}m+\sin\theta\left(\sqrt{-1} w+\frac{\mu_2+\mu_5}{2}m^2\cos^2\theta
-\frac{m^2\mu_4}{2}\right) B \non\\
&&\quad\ \  -\frac{m^2(\mu_5-\mu_2)}{2}B \sin^3\theta-\mu_2mwA\sin\theta=0,  \label{fix C}\\
&& D+\sin\theta \left( m^2+\sqrt{-1}\lambda_1\omega
\right)A+\frac{\lambda_2-\lambda_1}{2}\sqrt{-1}m B\sin\theta =0,
\label{fix D}
\eea
where \bea g(\theta) &=&
2\mu_1\cos^2\theta\sin^2\theta+(\mu_3+\mu_6)\cos^2\theta+\mu_4+(\mu_5-\mu_2)\sin^2\theta,
\label{g} \eea \be
     p(\theta) = \mu_2\sin^2\theta-\mu_3\cos^2\theta, \label{p} \ee
     and
     \bea
     q(\theta)&=&(\lambda_1+\lambda_2)\cos^2\theta+(\lambda_1-\lambda_2)\sin^2\theta  \non\\
&=&(\mu_2-\mu_3+\mu_5-\mu_6)\cos^2\theta+(\mu_2-\mu_3-\mu_5+\mu_6)\sin^2\theta.
\label{q}
 \eea

\bl\label{solution to p(theta) and q(theta)} Suppose that \eqref{lama1}, \eqref{lambda one is less than lambda two} are satisfied and $\mu_2\mu_3\geq 0$. There exists a unique real solution
 $\theta_0\in[0,\frac{\pi}{2}]$ to the equations
 \be
  \left\{\begin{array}{l}  p(\theta)=0, \\
   q(\theta)=0, \\
   \end{array}
  \right.\label{pqq}
 \ee if and only if Parodi's relation \eqref{lam2} holds.
  \el
\begin{proof}
If \eqref{lam2} holds, we have $q(\theta)=2p(\theta)$. Then it follows from \eqref{lambda one is less than lambda
two} and $\mu_2\mu_3\geq 0$ that the unique solution to \eqref{pqq} is given by
$$\theta_0=\arctan\left(\sqrt{\frac{\mu_3}{\mu_2}}\right)\in [0,\frac{\pi}{2}].$$

Conversely, suppose $\theta_0\in [0,\frac{\pi}{2}]$ is the solution to \eqref{solution to p(theta) and
q(theta)}. We discuss three subcases.

\textit{Case 1:} $\theta_0=0$. It follows from \eqref{p}
 that $\mu_3=0$. Then we infer from \eqref{q} that $\mu_2+\mu_5-\mu_6=0$ and as a result, $\mu_2+\mu_3=\mu_6-\mu_5$.

\textit{Case 2:} $\theta_0=\frac{\pi}{2}$. In this case we have $\mu_2=0$ and the proof is similar
to \textit{Case 1}.

\textit{Case 3:} $0 < \theta_0 < \frac{\pi}{2}$. In this case it is easy to see that $\mu_2\neq 0$, $\mu_3\neq 0$. Since
$\sin\theta_0 \neq 0$, $\cos\theta_0 \neq 0$, we deduce from \eqref{p}, \eqref{q} that
$$(\lambda_1-\lambda_2)\mu_3=-(\lambda_1+\lambda_2)\mu_2,$$
which combined with \eqref{lama1} yields \eqref{lam2}. The proof is complete.\qed
\end{proof}

In the remaining part of this section, we always suppose that
\eqref{lama1} is valid. We make the following
assumptions on the Leslie coefficients $\mu_2$, $\mu_3$, $\mu_5$, $\mu_6$:
 \bea
&&\mu_6 > 0, \ \mu_2 > 0,    \label{Le1} \\
&&\mu_5 < \mbox{min}\{\mu_2, \mu_6\},    \label{Le2} \\
&& \mu_3=\mu_6-\mu_5+\mu_2-\epsilon,\label{Le3a}\eea
where
 \be 0< \epsilon < \min\left\{\mu_6-\mu_5, 2\mu_2,
\frac{2(\mu_6-\mu_5)(\mu_2-\mu_5)}{4\mu_6-3\mu_5+3\mu_2} \right\}. \label{Le3}
 \ee
 Then we have
\begin{lemma}\label{unique solution}
Under the
assumptions \eqref{Le1}--\eqref{Le3}, the Leslie coefficients satisfy conditions \eqref{lama1a} and
\eqref{lambda one is less than lambda two}, but Parodi's relation \eqref{lam2} does not hold. Moreover, there exists
a unique solution $\theta_0 \in (0, \frac{\pi}{2})$ such that
$p(\theta_0) \neq 0$ and $q(\theta_0) = 0$.
\end{lemma}
\begin{proof}
It easily follows from \eqref{lama1}, \eqref{Le3a} and \eqref{Le3} that \eqref{lama1a} is satisfied. Besides, \eqref{Le3a} and \eqref{Le3} also imply that $$\mu_6-\mu_5 <
\mu_3-\mu_2 + 2\mu_2=\mu_2+\mu_3,$$ so Parodi's relation
\eqref{lam2} is not valid in this case. \eqref{lambda one is less
than lambda two} can be deduced from \eqref{Le2}, \eqref{Le3} and
\eqref{lama1a} in the sense that $$|\mu_5-\mu_6|=\mu_6-\mu_5 >
\mu_3-\mu_2 > 0.$$ Finally, \eqref{lama1a} and \eqref{Le1} yield that
$\mu_2\mu_3 > 0$. Therefore, we can deduce from Lemma \ref{solution to p(theta)
and q(theta)} that there exists an angle
\bea
 \theta_0&=&\arctan\sqrt{\frac{\mu_6-\mu_5+\mu_3-\mu_2}{\mu_6-\mu_5-\mu_3+\mu_2}}\non\\
 &=& \arctan\sqrt{\frac{2(\mu_6-\mu_5)-\epsilon}{\epsilon}}\in
\left(0,\frac{\pi}{2}\right)\label{theaa}
 \eea
 such that
 \be p(\theta_0)\neq 0\quad
\text{and}\quad q(\theta_0)=0.\label{pqt}
 \ee
The proof is complete.\qed
\end{proof}

We further assume that
 \bea
 && 0\leq \mu_1<  \frac14 (2\mu_6-\mu_5+\mu_2),  \label{Le4} \\
 && 0 \leq \mu_4<\frac12
(2\mu_6-\mu_5+\mu_2)\cos^2\theta_0,   \label{Le5}
 \eea
  where
$\theta_0$ is defined by \eqref{theaa}. Then we can state
the main result of this section:
 \begin{theorem}
 \label{proposition on currie}
Suppose that the Leslie coefficients $\mu_1, ...,\mu_6$ satisfy the assumptions
\eqref{Le1}--\eqref{Le3}, \eqref{Le4} and \eqref{Le5}. Then the linearized Ericksen--Leslie system
\eqref{lev}--\eqref{led2} admits unstable plane wave solutions.
\end{theorem}
\begin{proof}
Let $\theta_0$ be the angle obtained in Lemma \ref{unique solution} (cf. \eqref{theaa}).
Taking $\theta=\theta_0$ in the equations \eqref{third equation of Currie 74} and
\eqref{equation 2 of currie} and using \eqref{pqt}, we obtain that
\bea \left( m^2+\sqrt{-1}\lambda_1\omega
\right)A &=&0,
\label{equation I}  \\
m\omega p(\theta_0)A+\left(\frac{m^2g(\theta_0)}{2}-\sqrt{-1}\omega
\right)B&=&0. \label{equation II}
 \eea Choosing $$A=0,\ \   B=1\ \ \text{and}\  m\in \mathbb{R},\  m\neq 0,$$
 we deduce from \eqref{equation II} that
 \be
 \omega=-\sqrt{-1}\frac{m^2g(\theta_0)}{2}.\label{oo}
 \ee
 The (imaginary) constants $C$ and $ D$ are determined by \eqref{fix C} and \eqref{fix D}, respectively:
 \bea
 && C=\frac{m\sin\theta_0}{2}\left[ g(\theta_0)+\mu_2
-\mu_4+\mu_5\cos 2\theta_0\right]\sqrt{-1},  \non\\
&& D=-\frac{\lambda_2-\lambda_1}{2}\sqrt{-1}m \sin\theta_0.\non
 \eea
  \eqref{Le3} implies that
  \be
  \tan^2\theta_0= \frac{2(\mu_6-\mu_5)-\epsilon}{\epsilon}>\frac{2(2\mu_6-\mu_5+\mu_2)}{\mu_2-\mu_5}.\label{eee1}
  \ee
 Consequently,  we deduce from \eqref{Le4}, \eqref{Le5} and \eqref{eee1} that
    \bea &&
    g(\theta_0)\non\\
     &=& 2\mu_1\cos^2\theta_0\sin^2\theta_0+(\mu_3+\mu_6)\cos^2\theta_0+\mu_4\non\\
     && \ \ +(\mu_5-\mu_2)\sin^2\theta_0
 \non\\
 &=& 2\mu_1\cos^2\theta_0\sin^2\theta_0+\mu_4+(2\mu_6-\mu_5+\mu_2-\epsilon)\cos^2\theta_0\non\\
 && \ \ +(\mu_5-\mu_2)\sin^2\theta_0\non\\
 &\leq& \cos^2\theta_0\left(2\mu_1+\mu_4\sec^2\theta_0 +\left[(2\mu_6-\mu_5+\mu_2)+(\mu_5-\mu_2)\tan^2\theta_0\right]\right)\non\\
 &\leq& -\frac{(2\mu_6-\mu_5+\mu_2)}{2}\cos\theta_0     \non\\
 &<&0.
  \eea
 Thus, we obtain the following plane wave solutions $(d,v,\bar{\gamma}, \bar{p})$
  \bea
 d&=&0, \non\\
 v&=&(n-\nu\sin\theta_0)e^{\sqrt{-1}m\nu\cdot x- \frac{m^2g(\theta_0)}{2}t},\non\\
  \bar\gamma & =& Ce^{\sqrt{-1}m\nu\cdot x- \frac{m^2g(\theta_0)}{2}t}, \non\\
\bar{p}&=& De^{\sqrt{-1}m\nu\cdot x- \frac{m^2g(\theta_0)}{2}t},\non
 \eea
 which are unstable since $g(\theta_0)<0$. Here, we note that  $\nu-n\sin\theta_0\neq 0$ due to the fact $\theta_0\in (0, \frac{\pi}{2})$.
The proof is complete.\qed
\end{proof}
\br
Theorem \ref{proposition on currie} indicates that for the nematic
liquid crystal flow, if Parodi's relation \eqref{lam2} does not hold, the original Ericksen--Leslie system \eqref{the most primitive equ 1}--\eqref{the most
primitive equ 3} system can be \textit{(linearly) unstable}.\qed
\er

\section{Appendices}
\setcounter{equation}{0} In this section we provide some detailed
computations used in the previous sections.

\subsection{Least action principle}\label{cLAP}
The action functional takes the form
 \bea \mathbb{A}(x)
 =\int_0^T\int_{\Omega_0}
 \left[\frac12|x_t(X,t)|^2-\left(\frac12|\mathbb{F}^{-T}\nabla_X
 \mathbb{E}d_0(X)|^2+\mathcal{F}(\mathbb{E}d_0(X))\right)\right]J \, dX dt,\non
 \eea
 where $\Omega_0=Q$ is the original domain occupied by the material,
 $\mathbb{E}$ is the deformation tensor satisfying \eqref{mathcal E equ 1} and the Jacobian $J={\rm det}
 \mathbb{F}=1$. The above expression includes all the kinematic transport property
of the molecular director $d$. With different kinematic transport
relations, we will obtain different action functionals, even though
the energies may have the same expression in the Eulerian
coordinate.

We take any one-parameter family of volume preserving flow map
$$x^\epsilon(X,t) \quad \text{with} \ \ \ x^0=x,\ \  \left.\frac{d
x^\epsilon}{d\epsilon}\right|_{\epsilon=0}=y$$ and the
volume-preserving constraint $\nabla_x \cdot y=0$ (or
$J^\epsilon={\rm det}
 \mathbb{F}^\epsilon=1$). Applying the least action
principle, we have
$$\delta_x \mathbb{A}=\left.\frac{d\mathbb{A}(x^\epsilon)}{d\epsilon}\right|_{\epsilon=0}=0$$
such that
 \bea
0&=&\int_0^T\int_{\Omega_0} x_t\cdot y_t dXdt\non\\
&&\ -\int_0^T\int_{\Omega_0} \left(\mathbb{F}^{-T}\nabla_X
 \mathbb{E}d_0\right):\left[\left.\frac{d}{d\epsilon}\right|_{\epsilon=0}\left(\nabla_{x^\epsilon}
 d(x^\epsilon,t)\right)\right]dXdt\non\\
 &&-\int_0^T\int_{\Omega_0}f(\mathbb{E}d_0)\cdot\left(\Big.
 \frac{d\mathbb{E}^\epsilon}{d\epsilon}\Big|_{\epsilon=0} d_0\right)\
 dXdt\non\\
 &:=&  I_1+I_2+I_3,
 \label{variation of total energy}
 \eea
 where $\mathbb{E}^\epsilon=\mathbb{E}(x^\epsilon(X,t),t)$.
 Pushing forward to the Eulerian
coordinate, we have
  \bea I_1&=&-\int_{0}^T\int_{\Omega_0} x_{tt} \cdot y dXdt
  =-\int_{0}^T\int_{\Omega_t} \dot{v}\cdot y dxdt\non\\
  & = & -\int_{0}^T\int_{\Omega_t} (v_t+v\cdot\nabla v)\cdot y dxdt,
\label{variation of velocity}
 \eea
 where $\Omega_t$ is the domain occupied by the material at time
 $t$.

From the definition of $\mathbb{E}^\epsilon$, we have
 \be
 \left.\frac{d\mathbb{E}^\epsilon}{d\epsilon}\right|_{\epsilon=0} d_0
 =\left(\frac12(\nabla y-\nabla^{T}y)-\frac{\lambda_2}{2\lambda_1}(\nabla
 y+\nabla^{T}y)\right)\mathbb{E}d_0,
 \ee
  which implies that
   \bea
I_2 &=& -\int_0^T\int_{\Omega_0} \left(\mathbb{F}^{-T}\nabla_X
 \mathbb{E}d_0\right):
 \left(\left.\frac{d(\mathbb{F^\epsilon})^{-T}}{d\epsilon}\right|_{\epsilon=0}\nabla_X
 \mathbb{E}d_0\right)dXdt\non\\
 && -\int_0^T\int_{\Omega_0} \left(\mathbb{F}^{-T}\nabla_X
 \mathbb{E}d_0\right):\left[\mathbb{F}^{-T}\nabla_X
 \Big(\left.\frac{d\mathbb{E}^\epsilon}{d\epsilon}\right|_{\epsilon=0}d_0\Big)\right]dXdt\non\\
&=&  -\int_0^T\int_{\Omega_t} \nabla
 d : \left(-\nabla^T y \nabla d\right)dxdt \non\\
&&  -\int_0^T\int_{\Omega_t} \nabla
 d : \nabla \left[\Big(\frac{\nabla y-\nabla^{T}y}{2}-\frac{\lambda_2}{\lambda_1}\frac{\nabla
 y+\nabla^{T}y}{2}\Big)d\right] dxdt \non\\
 &=& -\int_0^T\int_{\Omega_t} \big[\nabla\cdot(\nabla d \odot\nabla
 d)\big]\cdot y dxdt \non\\
 && \ \  + \frac{1}{2}\Big(1-\frac{\lambda_2}{\lambda_1}\Big)\int_0^T
\int_{\Omega_t} \big[\nabla\cdot (\Delta d\otimes d)\big]\cdot y dxdt\non\\
&& \ \ - \frac{1}{2}\Big(1+\frac{\lambda_2}{\lambda_1}\Big)\int_0^T
\int_{\Omega_t} \big[\nabla\cdot ( d\otimes \Delta d)\big]\cdot y
dxdt, \label{PK}
 \eea
 and
 \bea  I_3&=&-\int_0^T\int_{\Omega_0} f(d) \cdot\left[\Big(\frac12(\nabla y-\nabla^{T}y)-\frac{\lambda_2}{2\lambda_1}(\nabla
 y+\nabla^{T}y)\Big)d\right] dXdt\non\\
 &=& \int_0^T\int_{\Omega_t}\left[ -\frac{1}{2}\Big(1-\frac{\lambda_2}{\lambda_1}\Big)\nabla\cdot(f(d)\otimes
d)\right]\cdot y \ dxdt\non\\
 & & +\int_0^T\int_{\Omega_t}\left[\frac{1}{2}\Big(1+\frac{\lambda_2}{\lambda_1}\Big)\nabla\cdot(d\otimes f(d))
\right]\cdot y \ dxdt. \label{expansion of variation of F(d)}
 \eea
 Inserting  \eqref{variation of velocity}, \eqref{PK} and \eqref{expansion of variation of F(d)}
 into \eqref{variation of total energy}, we arrive at
 \be
 \int_0^T\int_{\Omega_t} \big[v_t+v\cdot\nabla v+\nabla\cdot(\nabla d \odot\nabla
 d)-\nabla \cdot \tilde{\sigma}\big] \cdot y dxdt=0,\label{weak LAP}
 \ee
 where
 \be
 \tilde{\sigma}=-\frac{1}{2}\Big(1-\frac{\lambda_2}{\lambda_1}\Big) (\Delta d-f(d))\otimes d
 + \frac{1}{2}\Big(1+\frac{\lambda_2}{\lambda_1}\Big)  d\otimes (\Delta
 d-f(d)).\label{tildesigma}
 \ee
 Since $y$ is an arbitrary divergence free vector field,
 we formally derive the momentum equation
(Hamiltonian/conservative part) after integration by parts
 \be v_t+v\cdot\nabla
v=-\nabla P-\nabla \cdot (\nabla d\odot\nabla
d)+\nabla\cdot\tilde{\sigma}, \label{momentum equation from
consevative force}
 \ee
 where the pressure $P$ serves as a Lagrangian multiplier for the
 incompressibility of the fluid.

\subsection{Maximum dissipation principle}\label{cMDP}
Using the maximum dissipation principle \cite{O31,O31-2,O53}, we perform a variation on the
dissipation functional (half of the total rate of energy dissipation $\mathcal{D}$ \eqref{DIS2}) with respect
to the velocity $v$ in Eulerian  coordinates.
 If
$\delta_v\big(\frac12\mathcal{D}\big)$ is set to zero, we will get a
weak variational form of the dissipative force balance law
equivalent to conservation of momentum. Let $v^{\epsilon}=v+\epsilon
u$, where $u$ is an arbitrary regular function with $\nabla\cdot
u=0$. Then we have
 \bea
0&=&\delta_v\Big(\frac12\mathcal{D}\Big)= \frac12\frac{d\mathcal{D}(v^\epsilon)}{d\epsilon}\Big|_{\epsilon=0} \non\\
&=&\frac{\mu_4}{2}\int_Q \nabla v : \nabla u dx + \mu_1\int_Q
d_kA_{kp}d_p
d_i\frac{\nabla_iu_j+\nabla_ju_i}{2}d_jdx \non\\
&&-\lambda_1\int_Q\Big(d_t+v\cdot \nabla
d-\Omega\,d+\frac{\lambda_2}{\lambda_1}A\,d\Big)\non\\
&&\quad\quad  \quad \cdot \Big( u\cdot\nabla d -\frac{\nabla
u-\nabla^{T}u}{2}d+\frac{\lambda_2}{\lambda_1}\frac{\nabla
u+\nabla^{T}u}{2}d \Big)dx  \non\\
&&+\Big(\mu_5+\mu_6+\frac{(\lambda_2)^2}{\lambda_1}\Big)\int_Q
A_{ij}d_j\frac{\nabla_iu_k+\nabla_ku_i}{2}d_kdx\non\\
&:=& I_1+I_2+I_3+I_4.\non
 \eea
Using integration by parts, we get
 \bea
I_1&=&-\frac{\mu_4}{2}(\Delta v, u),
 \non\\
I_2 &=& -\mu_1\Big( \nabla\cdot\big[d^T A d(d\otimes d)\big], u\Big), \non \\
I_4
&=&-\frac12\Big(\mu_5+\mu_6+\frac{(\lambda_2)^2}{\lambda_1}\Big)\int_Q
\Big(u_k\nabla_i(d_kA_{ij}d_j)+u_i\nabla_k(A_{ij}d_jd_k)
     \Big)dx   \non\\
&=&-\frac12\Big(\mu_5+\mu_6+\frac{(\lambda_2)^2}{\lambda_1}\Big)
 \Big[\big( \nabla\cdot(d \otimes Ad)\, ,\,u\big)+\big(\nabla\cdot(Ad
\otimes d)\, ,\,u\big)\Big]. \non
 \eea
Using the transport equation \eqref{transport equ of d derived from
EVA} of $d$ and the incompressibility of $u$, we infer that
 \bea I_3&=&
 \left(\Delta d-f(d),
u\cdot\nabla
d-\frac{1}{2}\Big(1-\frac{\lambda_2}{\lambda_1}\Big)\nabla u d
+\frac{1}{2}\Big(1+\frac{\lambda_2}{\lambda_1}\Big)\nabla^T u d
\right)
   \non\\
&=&
 \left(u,
-\nabla F(d)+\nabla\cdot (\nabla d\odot \nabla
d)-\nabla\frac{|\nabla d|^2}{2} \right)
 \non\\
&&+  \Big(1-\frac{\lambda_2}{\lambda_1}\Big)\Big(u, \nabla\cdot
[(\Delta d-f(d))\otimes d] \Big)\non\\
&&\ \ -\Big(1+\frac{\lambda_2}{\lambda_1}\Big)\Big(u, \nabla\cdot
[d\otimes(\Delta
d-f(d))] \Big) \non\\
 &=& \big(u, \nabla\cdot(\nabla
d\odot \nabla d) \big) -\mu_2\big(u, \nabla\cdot(N\otimes d) \big)
-\mu_3\big(u, \nabla\cdot(d \otimes N) \big) \non\\
&&-\eta_5\big(u, \nabla\cdot(A\,d \otimes d) \big) -\eta_6\big(u,
\nabla\cdot(d \otimes A\,d) \big), \label{Onsager 3}
 \eea
 with the coefficients
 \bea
&& \mu_2=\frac{1}{2}(\lambda_1-\lambda_2), \quad \quad \quad
\mu_3=-\frac{1}{2}(\lambda_1+\lambda_2), \non\\
&&
\eta_5=\frac{1}{2}\left[\lambda_2-\frac{(\lambda_2)^2}{\lambda_1}\right],
\quad
\eta_6=-\frac{1}{2}\left[\lambda_2+\frac{(\lambda_2)^2}{\lambda_1}\right].
\label{meaning of eta 5 and eta 6}
 \eea
It follows from the above calculations that
 \bea
0=\frac12\frac{d\mathcal{D}}{d\epsilon}\Big|_{\epsilon=0} =\big(u,
\nabla\cdot(\nabla d\odot \nabla d) \big)-(u, \nabla\cdot\sigma).
 \eea
 The stress tensor $\sigma$ is given by
 \bea \sigma &=& \mu_1(d^T A d) d\otimes d+\mu_2N\otimes d
+\mu_3d\otimes N+\mu_4A+\tilde{\mu}_5Ad\otimes
d+\tilde{\mu}_6d\otimes Ad, \non
 \eea
 with constants
 \bea && \mu_2=\frac{1}{2}(\lambda_1-\lambda_2), \quad \quad
\mu_3=-\frac{1}{2}(\lambda_1+\lambda_2), \non\\
&& \tilde{\mu}_5=\frac{1}{2}(\lambda_2+\mu_5+\mu_6), \quad
\tilde{\mu}_6=\frac{1}{2}(-\lambda_2+\mu_5+\mu_6). \non
 \eea
 Since $u$ is an arbitrary function with $\nabla\cdot u=0$, we arrive
at the dissipative force balance equation
 \be 0= -\nabla P -\nabla\cdot(\nabla d\odot \nabla d)+
\nabla\cdot\sigma,
  \label{momentum equation from dissipative energy}
  \ee
  where the pressure $P$ serves as a Lagrangian multiplier for the
 incompressibility of the fluid.


\subsection{Computation on the time derivative of $\mathcal{A}(t)$}\label{cdA}

Using \eqref{e1}--\eqref{e3} and integration by parts, due to the
periodic boundary conditions, we have
\begin{eqnarray}
 && \frac{1}{2}\frac{d}{dt}\mathcal{A}(t)\non\\
&=&-(\Delta v, v_t)+(\Delta
d-f, \Delta d_t-f'(d)d_t)  \non\\
&=&(\Delta v, v\cdot\nabla v)+(\Delta v, \nabla d\Delta
d)-(\nabla\cdot\sigma, \Delta v)+\frac{1}{\lambda_1}\|\nabla(\Delta
d-f)\|^2
 \non\\&&-\big(\Delta d-f, \Delta(v\cdot\nabla d)\big)+\big(\Delta d-f, \Delta(\Omega
d)\big)-\frac{\lambda_2}{\lambda_1}\big(\Delta d-f, \Delta(Ad) \big) \non\\
&&+ \left( \Delta d-f, f'(d)\Big[\frac{1}{\lambda_1}(\Delta
d-f)+v\cdot\nabla d+\Omega d+\frac{\lambda_2}{\lambda_1}Ad
\Big]\right).
 \label{first time expansion of derivative of A}
\end{eqnarray}
First, we expand the third term on the right-hand side of
\eqref{first time expansion of derivative of A}:
 \bea
 && -(\nabla\cdot\sigma,
\Delta v)\non\\
&=&-\int_{Q}\nabla_j\sigma_{ij} \nabla_l\nabla_l v_{i}dx=-\int_{Q}\nabla_l\sigma_{ij} \nabla_l\nabla_j v_{i}dx \non\\
&=&-\mu_1\int_{Q}\nabla_l(d_kd_pA_{kp}d_id_j) \nabla_l\nabla_j
v_idx-\mu_4\int_{Q}\nabla_lA_{ij}
\nabla_l\nabla_j v_idx      \non\\
&&-\mu_2\int_{Q}\nabla_l(d_jN_i) \nabla_l\nabla_j
v_idx-\mu_3\int_{Q}\nabla_l(d_iN_j)\nabla_l\nabla_j
v_idx \non\\
&&-\mu_5\int_{Q}\nabla_l(d_jd_kA_{ki})\nabla_l\nabla_j
v_idx-\mu_6\int_{Q}\nabla_l(d_id_kA_{kj}) \nabla_l\nabla_j v_idx.
\non
 \eea
Using integration by parts and the fact that $\Omega$ is antisymmetric, we have
 \bea &&-\mu_1\int_{Q}\nabla_l(d_kd_pA_{kp}d_id_j)\nabla_l\nabla_j
v_i dx \non\\
&=& \mu_1\int_{Q} (d_kd_pA_{kp}d_id_j)\nabla_l\nabla_l(A_{ij}+\Omega_{ij}) dx
\non\\
&=& \mu_1\int_{Q} (d_kd_pA_{kp}d_id_j)\nabla_l\nabla_lA_{ij} dx\non\\
&=&-\mu_1\int_{Q}(d_kd_p\nabla_lA_{kp})^2dx
-\mu_1\int_{Q}A_{kp}\nabla_l(d_kd_p)d_id_j
\nabla_lA_{ij}dx
\non\\
&&-\mu_1\int_{Q}A_{kp}d_kd_p\nabla_l(d_id_j)\nabla_lA_{ij}dx.
\label{expansion of mu 1}
 \eea
By the incompressibility condition $\nabla \cdot v=0$, we see that
 \bea
-\mu_4\int_{Q}\nabla_l(A_{ij}) \nabla_l\nabla_j
v_idx&=&-\mu_4\int_{Q}\nabla_j(A_{ij}) \nabla_l\nabla_l v_idx\non\\
  &=& -\frac{\mu_4}{2}\|\Delta v\|^2.
\label{expansion of mu 4}
  \eea
Meanwhile,
 \bea &&-\mu_2\int_{Q}\nabla_l(d_jN_i) \nabla_l\nabla_j
v_idx-\mu_3\int_{Q}\nabla_l(d_iN_j)\nabla_l\nabla_j
v_idx \non\\
&=&\mu_2\int_{Q}d_jN_i
\Delta(A_{ij}+\Omega_{ij})dx+\mu_3\int_{Q}d_iN_j
\Delta(A_{ij}+\Omega_{ij})dx \non\\
&=&(\mu_2+\mu_3)\int_Q d_jN_i\Delta A_{ij}dx+(\mu_2-\mu_3)(N,\Delta
\Omega \,d),
 \label{expansion of mu 2 and 3}
 \eea
 and
 \bea &&-\mu_5\int_{Q}\nabla_l(d_jd_kA_{ki})
\nabla_l\nabla_j v_idx-\mu_6\int_{Q}\nabla_l(d_id_kA_{kj})
\nabla_l\nabla_j v_i dx   \non\\
&=&\mu_5\int_{Q}d_jd_kA_{ki}
\Delta(A_{ij}+\Omega_{ij})dx+\mu_6\int_{Q}d_jd_kA_{ki}
\Delta(A_{ij}-\Omega_{ij})dx   \non\\
&=&(\mu_5+\mu_6)\int_{Q}d_jd_kA_{ki} \Delta
A_{ij}dx+(\mu_5-\mu_6)\int_{Q}d_jd_kA_{ki}
\Delta\Omega_{ij}dx  \non\\
&=&-(\mu_5+\mu_6)\int_{Q}d_jd_k\nabla_lA_{ki}
\nabla_lA_{ij}dx-(\mu_5+\mu_6)\int_{Q}\nabla_ld_jd_kA_{ki}
\nabla_lA_{ij}dx   \non\\
&&-(\mu_5+\mu_6)\int_{Q}d_j\nabla_ld_kA_{ki}
\nabla_lA_{ij}dx+(\mu_5-\mu_6)\big(Ad, \Delta\Omega d \big)    \non\\
&=&-(\mu_5+\mu_6)\int_{Q}|d_j\nabla_lA_{ji}|^2dx-(\mu_5+\mu_6)\int_{Q}\nabla_ld_jd_kA_{ki}
\nabla_lA_{ij}dx
\non\\
&&-(\mu_5+\mu_6)\int_{Q}d_j\nabla_ld_kA_{ki}
\nabla_lA_{ij}dx+(\mu_5-\mu_6)\big(Ad, \Delta\Omega d \big).
\label{expansion of mu 5 and mu 6}
 \eea
 Next, using the $d$ equation \eqref{e3}, we have
\bea && (\Delta d-f, \Delta(\Omega\,d)) \non\\
 &=&(\Delta d-f, \Delta\Omega\,
d)+2\int_Q(\Delta d_i-f_i) \nabla_l\Omega_{ij}\nabla_l d_j
dx\non\\
&& \ \ +(\Delta d-f,
\Omega\Delta d)   \non\\
&=&-\lambda_1\int_{Q}d_jN_i\Delta\Omega_{ij}\, dx-\lambda_2\big(Ad,
\Delta\Omega d \big)\non\\
&&\ \  +2\int_Q(\Delta d_i-f_i)
\nabla_l\Omega_{ij}\nabla_l d_j dx+(\Delta d-f, \Omega\Delta d) \non\\
&=& -\lambda_1(N, \Delta\Omega\,d)-\lambda_2\big(Ad, \Delta\Omega d
\big) -\int_Q\nabla_l (\Delta d_i-f_i) \Omega_{ij}\nabla_l d_j dx
 \non\\
 &&\ \ +\int_Q(\Delta d_i-f_i) \nabla_l\Omega_{ij}\nabla_l d_j dx,
 \label{expansion of laplace omega d}
 \eea
 \bea
 && -\frac{\lambda_2}{\lambda_1}\big(\Delta d-f, \Delta(Ad) \big)\non\\
&=&\lambda_2 \big(N, \Delta(Ad)
\big)+\frac{(\lambda_2)^2}{\lambda_1}\big(Ad,
\Delta(Ad) \big)   \non\\
&=&\lambda_2\int_Q N_i\Delta A_{ij}d_jdx+2\lambda_2\int_Q
N_i\nabla_l A_{ij}\nabla_l d_j dx+\lambda_2(N, A\Delta d)\non\\
&& -\frac{(\lambda_2)^2}{\lambda_1}\int_{Q}|\nabla_l
(A_{ij}d_j)|^2dx. \label{expansion of laplace A d}
 \eea

 \textit{Special Cancellations.} (i) Due to  \eqref{lama1}, the first term on the right-hand side of
\eqref{expansion of laplace omega d} cancels with the second term of
the right-hand side of \eqref{expansion of mu 2 and 3} and the
second term on the right-hand side of \eqref{expansion of laplace
omega d} cancels with the fourth term of the right-hand side of
\eqref{expansion of mu 5 and mu 6}. (ii) By Parodi's relation
\eqref{lam2}, the first term of the right-hand side of
\eqref{expansion of laplace A d} cancels with the first term of the
right-hand side of \eqref{expansion of mu 2 and 3}.

Concerning the fifth term on the right-hand side of \eqref{first
time expansion of derivative of A}, using the incompressibility of
$v$, the fact $\nabla d \cdot f(d)  =\nabla \mathcal{F}(d)$ and integration by
parts, we obtain
\bea
&&-(\Delta d-f, \Delta((v\cdot\nabla) d)) \non\\
&=&-(\Delta d-f, \Delta v\cdot\nabla d)-2\int_Q(\Delta
d_i-f_i)\nabla_l v_j\nabla_l\nabla_j d_idx-(\Delta d-f,
v\cdot\nabla\Delta d)\non\\
&=&-(\Delta v, \nabla d\Delta d)+2\int_Q\nabla_j (\Delta
d_i-f_i)\nabla_l v_j\nabla_ld_idx-(\Delta d-f, v\cdot\nabla f).\non
 \eea
Hence,
  \bea &&-\big(\Delta d-f, \Delta(v\cdot\nabla
d)\big)+(\Delta v, \nabla d\Delta d)\non\\
&&\;\; +\left(\Delta d-f, f'(d)\Big[\frac{1}{\lambda_1}(\Delta
d-f)+v\cdot\nabla d-\Omega\, d+\frac{\lambda_2}{\lambda_1}A\,d \Big]\right) \non\\
&=&\frac{1}{\lambda_1}\int_{Q}f'(d)|\Delta d-f|^2dx- \left(\Delta
d-f, f'(d)\Big(\Omega\, d-\frac{\lambda_2}{\lambda_1}A\,d\Big)
\right)\non\\
&& \;\;+2\int_Q\nabla_j (\Delta d_i-f_i)\nabla_l
v_j\nabla_ld_idx-(\Delta d-f, v\cdot\nabla f). \label{expansion of
the final three terms of A(t)}
    \eea
Collecting the above calculations together, we conclude that
\eqref{second time expansion of derivative of A(t)} holds.


%

\section*{Acknowledgements}
H. Wu was partially supported by NSF of
China 11001058, Specialized Research Fund for
the Doctoral Program of Higher Education and "Chen Guang" project supported by Shanghai
Municipal Education Commission and Shanghai Education Development
Foundation. C. Liu and X. Xu were partially supported by NSF grants
DMS-0707594 and DMS-1109107. This project began during a long term visit of X. Xu and C. Liu to
IMA of University of Minnesota, whose hospitality is gratefully
acknowledged. They would like to thank Professors C. Calderer, C.
Doering, D. Kinderlehrer, C.-M. Li, F.-H. Lin, E. Titi and C.-Y.
Wang for many helpful discussions. 

\end{document}